\documentclass[11pt]{article}
\usepackage{amsmath,amssymb,algorithm,xcolor,amsthm,a4wide,dsfont, hyperref, pgfplots}
\usepackage{algpseudocode}
\usepackage[framemethod=TikZ]{mdframed}
\newcounter{theo}[section]

\renewcommand{\thetheo}{\arabic{section}.\arabic{theo}}
\newenvironment{theo}[2][]{%
	\refstepcounter{theo}%
	\ifstrempty{#1}%
	{\mdfsetup{%
			frametitle={%
				\tikz[baseline=(current bounding box.east),outer sep=0pt]
				\node[anchor=east,rectangle,fill=blue!30]
				{\strut Theorem~\thetheo};}}
	}%
	{\mdfsetup{%
			frametitle={%
				\tikz[baseline=(current bounding box.east),outer sep=0pt]
				\node[anchor=east,rectangle,fill=blue!30]
				{\strut Theorem~\thetheo:~#1};}}%
	}%
	\mdfsetup{innertopmargin=10pt,linecolor=blue!30,%
		linewidth=2pt,topline=true,%
		frametitleaboveskip=\dimexpr-\ht\strutbox\relax
	}
	\begin{mdframed}[]\relax%
		\label{#2}}{\end{mdframed}}
\newcounter{lem}[section] 
\renewcommand{\thelem}{\arabic{section}.\arabic{lem}}
\newenvironment{lem}[2][]{%
	\refstepcounter{lem}%
	\ifstrempty{#1}%
	{\mdfsetup{%
			frametitle={%
				\tikz[baseline=(current bounding box.east),outer sep=0pt]
				\node[anchor=east,rectangle,fill=green!20]
				{\strut Lemma~\thelem};}}
	}%
	{\mdfsetup{%
			frametitle={%
				\tikz[baseline=(current bounding box.east),outer sep=0pt]
				\node[anchor=east,rectangle,fill=green!20]
				{\strut Lemma~\thelem:~#1};}}%
	}%
	\mdfsetup{innertopmargin=10pt,linecolor=green!20,%
		linewidth=2pt,topline=true,%
		frametitleaboveskip=\dimexpr-\ht\strutbox\relax
	}
	\begin{mdframed}[]\relax%
		\label{#2}}{\end{mdframed}}
\newcounter{corr}[section] 
\renewcommand{\thecorr}{\arabic{section}.\arabic{corr}}
\newenvironment{corr}[2][]{%
	\refstepcounter{corr}%
	\ifstrempty{#1}%
	{\mdfsetup{%
			frametitle={%
				\tikz[baseline=(current bounding box.east),outer sep=0pt]
				\node[anchor=east,rectangle,fill=orange!40]
				{\strut Corollary~\thecorr};}}
	}%
	{\mdfsetup{%
			frametitle={%
				\tikz[baseline=(current bounding box.east),outer sep=0pt]
				\node[anchor=east,rectangle,fill=orange!40]
				{\strut Corollary~\thecorr:~#1};}}%
	}%
	\mdfsetup{innertopmargin=10pt,linecolor=orange!40,%
		linewidth=2pt,topline=true,%
		frametitleaboveskip=\dimexpr-\ht\strutbox\relax
	}
	\begin{mdframed}[]\relax%
		\label{#2}}{\end{mdframed}}
	
\newcounter{cond}[section] 
\renewcommand{\thecond}{\arabic{section}.\arabic{cond}}
\newenvironment{cond}[2][]{%
	\refstepcounter{cond}%
	\ifstrempty{#1}%
	{\mdfsetup{%
			frametitle={%
				\tikz[baseline=(current bounding box.east),outer sep=0pt]
				\node[anchor=east,rectangle,fill=red!20]
				{\strut Condition~\thecond};}}
	}%
	{\mdfsetup{%
			frametitle={%
				\tikz[baseline=(current bounding box.east),outer sep=0pt]
				\node[anchor=east,rectangle,fill=red!20]
				{\strut Condition~\thelem:~#1};}}%
	}%
	\mdfsetup{innertopmargin=10pt,linecolor=red!20,%
		linewidth=2pt,topline=true,%
		frametitleaboveskip=\dimexpr-\ht\strutbox\relax
	}
	\begin{mdframed}[]\relax%
		\label{#2}}{\end{mdframed}}

\author{
Coralia Cartis \thanks{The order of the authors is alphabetical; the second author (Sadok Jerad) is the primary contributor.} \thanks{Mathematical Institute, Woodstock Road, University of Oxford, Oxford, UK, OX2 6GG. \href{mailto:coralia.cartis@maths.ox.ac.uk}{coralia.cartis@maths.ox.ac.uk}. 
}
\,\,\,\,
Sadok Jerad \footnotemark[1] \thanks{Mathematical Institute, Woodstock Road, University of Oxford, Oxford, UK, OX2 6GG. \href{mailto:sadok.jerad@maths.ox.ac.uk}{sadok.jerad@maths.ox.ac.uk}}
\,\,\,\,
Karl Welzel \thanks{Mathematical Institute, Woodstock Road, University of Oxford, Oxford, UK, OX2 6GG.
\href{mailto:karl.welzel@maths.ox.ac.uk}{karl.welzel@maths.ox.ac.uk} 
} 
}
\title{On Global Rates for Regularization Methods based on Secant Derivative Approximations}

\newtheorem{assumption}{Assumption}
\newcommand{\Ren}{\mathbb{R}^n}
\newcommand{\eqdef}{\stackrel{\rm def}{=}}
\newcommand{\iibe}[2]{\{ #1, \ldots, #2 \}}
\newcommand{\powppminsone}[1]{#1^\sfrac{p+1}{p-1}}

\DeclareMathOperator*{\argmin}{arg\,min}
\newcommand{\twoopptwo}{2^\sfrac{2}{p-1}}
\newcommand{\indica}[1]{\mathds{1}_{#1}}
\newcommand{\bigsum}{\displaystyle \sum}
\newcommand*{\OFFOplus}{\texttt{OFFO\textsuperscript{+}}}

\numberwithin{equation}{section}

\newcommand{\sfrac}[2]{{\scriptstyle \frac{#1}{#2}}}
\begin{document}
	\maketitle
\begin{abstract}
An inexact {and globally convergent} framework for high-order adaptive regularization methods is presented, in which approximations may be used for the $p$th-order tensor, based on lower-order derivatives. Between each recalculation of the $p$th-order derivative approximation, a high-order secant equation can be used to update the $p$th-order tensor as proposed in {(Karl Welzel and Raphael A Hauser, Approximating higher-order derivative tensors using secant updates, SIAM J.Optim, 34(1), 2024)} or the approximation can be kept constant in a lazy manner. When refreshing the $p$th-order tensor approximation after $m$ steps, an exact evaluation of the tensor or a finite difference approximation can be used with an explicit discretization stepsize. For all the newly adaptive regularization variants, we prove an $\mathcal{O}\left( \max[ \epsilon_1^{-(p+1)/p}, \, {\epsilon_2^{-(p+1)/(p-1)}}  ] \right)$ bound on the number of iterations needed to reach an $(\epsilon_1, \, \epsilon_2)$ second-order stationary points. Discussions on the number of oracle calls for each introduced variant are also provided. 
When  $p=2$, we obtain a second-order method
%\footnote{\correct{The method is considered second-order since it utilizes (inexact) Hessian information. }} 
that uses quasi-Newton approximations with an $\mathcal{O}\left(\max[\epsilon_1^{-3/2}, \, \, \epsilon_2^{-3}]\right)$ iteration bound to achieve approximate second-order stationarity. 
{Numerical illustrations for the case $p=3$ are provided in both the deterministic and noisy settings  showcasing the merits of secant updates for approximating third-order information, as well as the robustness of our proposed method even in noisy cases.}
\end{abstract}

\section{Introduction}
In this paper, we consider the unconstrained nonconvex optimization problem,
\begin{equation}\label{minf}
	\min_{x \in \Ren} f(x)
\end{equation}
where $f$ is $p$ times continuously differentiable, with $p\geq 2$. Recent research \cite{BirgGardMartSantToin17,CartGoulToin20b,CarGouldToint19,Nesterov2019} has shown that improved/optimal worst-case complexity bounds can be obtained when high-order derivative information of the objective function is used alongside adaptive regularization techniques. Namely, under Lipschitz continuity assumptions on $\nabla_x^p f$, the \texttt{ARp} algorithm requires
no more than $\mathcal{O}\left(\epsilon^{-(p+1)/p}\right)$ evaluations of $f$ and its derivatives to compute an approximate first-order local minimizer that satisfies $\| \nabla_x^1 f(x_k)\| \leq \epsilon$, which is optimal for this
function class \cite{Carmon2019a,Carmon2019b}. Moreover, local convergence rates have been analyzed in strongly convex neighborhoods \cite{Doikov2021} when the regularization parameter is sufficiently large. 

However, despite their theoretical improvements, higher-order methods ($p \geq 3$) require exact access to  $p$th-order information, which may be unavailable or too costly to compute. Interpolation models have been proposed to construct an efficient high-order approximation to third and fourth order derivatives as  in \cite{Schnabel1991,Schnabel1984}.
Aside from performing interpolation, there are two main schemes for approximating a given derivative with lower-order information. Firstly, finite differences with an appropriately chosen step size can be employed to approximate the desired derivative; for example, one can approximate the Hessian matrix by using either gradient or function values. In the latter case, one falls back on the derivative-free framework \cite{ConnSchVic09}. The former case has also been considered theoretically and in practice, see both \cite{Cartis2022-wb,Grapiglia2021,doikov2023zerothorder} and the references therein.

A popular alternative is quasi-Newton methods, where an approximate curvature matrix is constructed from gradient information only. Several useful schemes have been developed  such as BFGS  \cite{Fletcher1970,Goldfarb1970,BROYDEN1970,Shanno1970}, DFP \cite{Fletcher1963,Davidon59}, PSB \cite{Broyden1965}, SR1 \cite{ConnGouldToint91,Byrd1996}, L-BFGS \cite{LiuNoce89} to name just a few; see   \cite{NoceWrig06} for more details. However, the theoretical understanding of quasi-Newton variants in terms of global rates of convergence on general objectives has not been fully understood. Guarantees of local asymptotic superlinear rates were first established in \cite{BROYDEN1973,Boggs1982}, as well as global convergence for convex quadratic objectives \cite{NoceWrig06}, but securing bounds for their global rates of convergence in more general settings has been under-researched until recently. A new series of works \cite{Rodomanov2021,Rodomanov2021b,JinAryanetal22,JinAryan22,JinJiang24,Ye2022}  filled this gap by providing non-asymptotic rate analyses in the (strongly) convex case. Other works have focused on providing global analyses of trust-region methods with unbounded Hessian approximations, and applying these results to  quasi-Newton approximations \cite{Powell2009,leconte2023complexity,Diouaneetal24}. Due to the weak (but realistic) assumptions made on the Hessian approximations in these papers, expectedly, the bounds obtained are worse than the optimal bounds (for second- or higher-order) regularization methods; our approach here can be seen as a way to analyze re-started versions of similar algorithms, which then can control the size of these Hessian approximations and hence allow us to improve the bounds.
In \cite{Jiangal24, Scieur24}, global convergence theory is developed for quasi-Newton methods in the nonconvex case. 
The analysis in \cite{Jiangal24} focused on an online-learning quasi-Newton variant, obtaining a mildly-dimension dependent bound, while \cite{Scieur24} studies a subspace variant of cubic regularization where the subspace depends on differences of past gradients and proved a suboptimal complexity rate.

A different line of work made progress in the development and theory of inexact second- and high-order methods. For Hessian-based schemes, lazy variants have been developed where the approximation of the second-order derivative is updated every $m$ steps with the exact Hessian \cite{DoikChayJag23}, or with a derivative-free approximation \cite{doikov2023zerothorder}. Extension to a min-max setting has also been proposed in \cite{chen2025secondorder}. Another research direction provides  OFFO \footnote{(Objective Function Free Optimization)} adaptive regularization variants, to ensure there is no need to evaluate the objective in order to adjust algorithm parameters and make progress. These objective-free approaches were motivated by the success of adaptive gradient methods for modern machine learning \cite{KingBa15, TielHint12, Duchi11adagrad,Grapiglia2022}. In \cite{GraJerToin24}, it was shown that a stochastic adaptive  regularization algorithm with a specific update rule, which does not compute the function value or a proxy, can achieve the optimal complexity bounds of $p$th-order methods proved in \cite{BirgGardMartSantToin17,Carmon2019a}, provided that the error on the tensors is controlled by the step sizes of the last $m$ steps. It should be noted that this new condition mitigates the bottleneck of conventional adaptive regularization, in which the errors on the tensors are typically controlled by the current step; the latter characteristic renders inexact algorithms implicit and requires {\it a posteriori} verification, as discussed in \cite[Chapter~13]{Cartis2022-wb}.  

Directly relevant to our work here is \cite{Welzel2024}, where a new class of high-order quasi-Newton variants 
(High-Order Secant Update (HOSU))
{{was proposed for general nonconvex problems. By deriving and exploiting a secant equation for the $p$th-order derivative and  using that quasi-Newton updates are least-change tensor modifications satisfying this secant equation, the authors of \cite{Welzel2024}  are able to extend the classical PSB \cite{Broyden1965} and DFP \cite{Fletcher1963,Davidon59} formulas from second-order approximations to $p$th-order tensor approximations;}} they also {{showed the local convergence of}} the approximate tensor to the true derivative. {{However, no globally convergent algorithmic framework, nor any global analysis, that includes these tensor approximation updates, was proposed in 
		\cite{Welzel2024}.}}

Our current work sits at the crossroads of these different paradigms where we build an adaptive regularization framework where the $p$th-order tensor may be approximated using the lower-order ones. We use the update rule in \cite{GraJerToin24} to handle the arising inexactness and we propose different approximations of the $p$th-order tensor as follows.

\begin{itemize}
	\item {{We analyse a lazy derivative(s) update,}} which extends the lazy Hessian paradigm \cite{DoikChayJag23,doikov2023zerothorder} to the high-order algorithms. In our case, we will recompute the $p$th-order tensor exactly every $m$ iterations.
	
	\item {{We propose an algorithmic framework for nonconvex smooth optimization that employs the High-Order Secant Update (HOSU) recently proposed in \cite{Welzel2024}, with (existing) global regularization and OFFO strategies, in order to achieve both a lower computational cost per iteration of tensor methods and their global convergence.  We prove, for the first time, a global rate analysis for such concrete, secant-based tensor methods. }}

	\item Our results apply to the case when $p=2$, where we consider the standard PSB and DFP quasi-Newton methods.
	\item All of our algorithms are fully adaptive and objective function free. Thus, we do not require function evaluations to ensure convergence.
	\item A  finite difference framework, where we use the $(p-1)$th order derivative to compute an approximation of the $p$th-order tensor every $m$ steps is also covered by our results. Since the required error depends on the past steps of the algorithm, the length  of the finite difference step is explicit and depends only past iterates' information and algorithm hyper-parameters.
	\item {We perform preliminary numerical tests for the variously proposed inexact third-order methods. Experiments on both the deterministic and noisy cases highlight the merit of the High-Order Secant Update (HOSU), the approximation's refreshment for every $m$ iterations for well chosen $m$ and the usefulness of the OFFO update rule in the inexact noisy case. 
		% and the robustness of our method for the noisy case as showcased in experiments   
	}
\end{itemize} 
In particular, when specifying our result for the case $p=2$, we are able to reach an $(\epsilon_1, \epsilon_2)$ second-order stationary point  in $\mathcal{O}\left(n \max[ \epsilon_1^{-3/2}, \epsilon_2^{-3}]\right)$ calls of the gradient for second-order methods that approximate Hessians with finite differences every $m$ steps with the use of quasi-Newton PSB or DFP approximations on the remaining iterations. When $p=3$, we devise a ``quasi-Newton high-order'' method that does not compute the third-order tensor but approximates it with the (HOSU) scheme in \cite{Welzel2024} and with
finite differences of the Hessian every $m$ steps. For this method that uses only  Hessian and gradient evaluations, we obtain a complexity bound  $\mathcal{O}\left(n \max[ \epsilon_1^{-4/3}, \epsilon_2^{-1/2}]\right)$. {Building on the OFFO paradigm for high-order methods introduced in \cite{OFFO-ARp,GraJerToin24}, we derive complexity guarantees for finding second-order stationary points while employing  secant-based approximations of the $p$th-order derivative tensor \cite{Welzel2024}.} 
Our analysis will be similar in some aspects to that of \cite{GraJerToin24}: we will use some of its results and the proofs in this paper will remain in the same spirit. The common elements of the analysis stem from the use of the same update rule for the regularization parameter, which has been shown to handle inexactness in the tensor derivatives.  However, we improve upon \cite{GraJerToin24} by removing a restrictive condition on the negative curvature of the approximate Hessian, which would otherwise hinder the application of our algorithm in the quasi-Newton and secant tensor cases. Our focus here is also different or more specific than in \cite{GraJerToin24}: we are in the deterministic case and aim to find explicit second- and higher-order derivative approximations from lower order derivatives, of quasi-Newton and secant type, that achieve fast global rates of convergence. Before stating our assumptions and  algorithmic framework, we detail the notations used in subsequent developments. 

\paragraph*{Some key concepts and notations} The following notations will be used throughout the paper:  The symbol $\|.\|$
denotes the Euclidean norm for vectors in $\Ren$;  $\lambda_{\min}$ and $\lambda_{\max}$ denote respectively the leftmost eigenvalue and rightmost eigenvalue of a symmetric matrix; $x^\intercal y$ denotes the Euclidean inner product between two vectors $x$, $y$ in $\Ren$. {$\mathds{1}_{A}$ denotes the indicator function for an event $A$.}

As we will be also concerned with high-order methods, we also define the following.
Let $\mathbb{R}^{\otimes p_n}$ denote the space of  multi-linear maps from $\mathbb{R}^{n} \times \mathbb{R}^{n} \dots \mathbb{R}^{n}$ to $\mathbb{R}$. For $T \in \mathbb{R}^{\otimes p_n}$ and $s \in \mathbb{R}^n$, we have $T[s]^p = T[s,s, \dots ,s] \in \mathbb{R}^n$. For $v_1, v_2, \dots v_n \in \mathbb{R}^n$, the tensor $T = v_1 \otimes v_2 \dots \otimes v_n$ is defined as $T[s_1, s_2, \dots s_p] = \prod_{i=1}^{p} v_i^\intercal s_i$. For $W$ a matrix $\mathbb{R}^{n \times n}$, we define $T[W]^p$ as 

\begin{equation}\label{Twdef}
	(T[W]^p)[s_1, s_2, \dots s_p] = T[Ws_1, Ws_2, \dots Ws_p ].
\end{equation}
\noindent
For any $T \in \mathbb{R}^{\otimes p_n}$ and any permutation $ \sigma \in S_p $, let $ \sigma(T) \in \mathbb{R}^{\otimes p_n}$ be defined by
\[
\sigma(T)[s_1, \dots, s_p] = T[s_{\sigma(1)}, \dots, s_{\sigma(p)}].
\]
If  $\sigma(T) = T$ for all $\sigma \in S_p$, then $T$ is called symmetric. The space of all symmetric \( p \)-tensors is denoted \( \mathbb{R}^{\otimes p_n}_{{sym}} \). The projection of \( \mathbb{R}^{\otimes p_n} \) onto \( \mathbb{R}^{\otimes p_n}_{{sym}} \) is given by
\[
P_{{sym}}(T) = \frac{1}{p!} \sum_{\sigma \in S_p} \sigma(T).
\]

\noindent
As in matrices, we can decompose the action of tensor on a canonical basis, a tensor $T \in \mathbb{R}^{\otimes p_n}$ therefore writes as 

\begin{equation}
	T = \sum_{i_1,\dots i_p = 1}^{n} t_{i_1,i_2, \dots i_p} e_{i_1} \otimes  \dots \otimes e_{i_p} \text{ with } t[i_1, i_2, \dots i_p] = T[e_{i_1},  \dots e_{i_p} ],
\end{equation}
By this representation as p-dimensionnal array, we can define an associate Frobenius scalar product and a norm $\langle T_1, T_2 \rangle_F$ and $\|T\|_F$. For a positive definite matrix $W$, we define $\|.\|_{2,W}$ as 
\begin{equation}\label{W2weighted}
	\|T\|_{W} = \max_{\|s_i\| = 1, 1\leq i \leq p} |T[W s_1, \dots W s_p]|.
\end{equation}
\noindent
For $W = I_n$, we retrieve the standard 2-norm $\|.\|_2$ which will be denoted as $\|.\|$ for the sake of simplicity and will be used to denote the 2-norm of any tensor.

Before proceeding with the analysis, we state a simple lemma that allow to perform the comparison between the weighted tensor norm $\|.\|_W$ and the standard one $\|.\|$.

\begin{lem}{CompairsonWnorm}
Let $W$ be symmetric positive definite matrix. Then, we have that 
\begin{equation}\label{weightednormcomparison}
\lambda_{\min}(W)^p \|T\|_W	\leq \|T\| \leq \lambda_{\max}(W)^p \|T\|_W
\end{equation}
\end{lem}

\begin{proof}
	Let $s_i$ and $1 \leq i \leq p$ such that $\|s_i\|=1$. Then, we have that 
	\[
	\frac{|T[W]^p [ s_1, \cdots s_p ]|}{\prod_{i=1}^p \|Ws_i\|} \leq \|T\|_W.
	\]
	Now rearranging the last inequality and taking the the max over $\|s_i\|=1$, 
	\[
	\|T\| \leq \max_{\|s_i\|=1, 1 \leq i \leq p}{\prod_{i=1}^p \|Ws_i\|} \|T\|_W \leq \lambda_{\max}(W)^p \|T\|_W.
	\]
	which gives the r.h.s of \eqref{weightednormcomparison}. Now for the other side, applying the found inequality with $\widehat{T} = T[W^p]$ and the $\|.\|_{W^{-1}}$ norm, we have 
	\[
	\|T\|_W = \|T[W^p]\| \leq \lambda_{\max}(W^{-1})^p  \|T\| 
	\]
	Now rearranging the last inequality yields the second part of \eqref{weightednormcomparison}.
	\end{proof}

We now detail the outline of our paper. Section~\ref{presentation} describes our algorithmic framework and discusses the condition imposed on the inexact $p$th-order tensor. This condition enables various adaptive regularization algorithms, which will be described later in the paper. Section~\ref{analysis} analyses our algorithm and carefully treats the negative curvature in order to extend the analysis performed in \cite{GraJerToin24}.  This will enable us to remove any conditions on the quasi-Newton approximations. Section~\ref{Propupdates} details the different algorithmic variants covered by our analysis. Specifically, we present a lazy variant for high-order tensors, a finite difference variant with an explicit step size and  the incorporation of High-Order Secant Updates \cite{Welzel2024}. {Illustrative numerical experiments are developed in Section~\ref{numeric-s} for the third-order case for both the deterministic and noisy cases. }  Section~\ref{conclu-s} presents some conclusions and outlines future research directions.

\section{A $p$th-order Adaptive Regularization Framework with Inexact $p$th-order Tensor}\label{presentation}

In this section, we present our algorithmic framework and detail the impact of our inexact condition on the $p$th-order tensor. Specifically, we start by describing our algorithmic framework in Section~\ref{algdescr-subsec} (Algorithm~\ref{AlgolazyOFFO}) and compare it with related adaptive regularization methods \cite{GraJerToin24,BirgGardMartSantToin17}. In the later Section~\ref{subsec-errbound}, we analyse the impact of the  condition on the $p$th-order tensor approximation (Condition~\ref{Conditionpth}) translates at each iteration and perform a comparison with the closely related work of \cite{GraJerToin24}.  

\subsection{Algorithmic description} \label{algdescr-subsec}

We now state our conditions on the function $f$ and describe a generic  algorithmic
framework.

\begin{assumption}\label{assum1}
	Assume	$f$ in \eqref{minf} is $p$ times continuously differentiable with $p\geq2$.
\end{assumption}

At an iterate $x_k$, we approximately minimize a $p$th  
degree regularized model $m_k(s)$ of $f(x_k+s)$ of the form
\[ 
T_{f,p}(x_k,s) + \frac{\sigma_k}{(p+1)!} \|s\|^{p+1},
\]
where $T_{f,p}(x,s)$ is the $p$th-order Taylor expansion of
functional $f$ at $x$ truncated at order $p$, that is,
\begin{equation}\label{taylor}
	T_{f,p}(x,s) \eqdef f(x) + \sum_{i=1}^p \frac{1}{i!} {\nabla_x^i f(x)}[s]^i. 
\end{equation}

In our case, we  consider an approximate Taylor approximation where exact access to first up to $(p-1)$th derivative is assumed, while for the $p$th-order tensor, we use an approximation $T_k$, of the true tensor,
\begin{equation}\label{approxtaylor}
	\overline{T_{f,p}}(x,s) \eqdef f(x) + \sum_{i=1}^{p-1} \frac{1}{i!} {\nabla_x^i f}(x)[s]^i + \frac{T_k[s]^p}{(p+1)!},
\end{equation}
and the model $m_k$ is then,
\begin{equation}\label{model}
	m_k(s) \eqdef \overline{T_{f,p}}(x_k,s) + \frac{\sigma_k}{(p+1)!} \|s\|^{p+1}.
\end{equation}

The $\frac{\sigma_k}{(p+1)!} \|s\|^{p+1}$ guarantees that $m_k(s)$ is bounded below and
thus makes the procedure of finding a step $s_k$ by (approximately) minimizing $m_k(s)$ well-defined. 
{We define a $p$th-order method as an iterative scheme that employs a $p$th degree  polynomial (regularized by $(p+1)$st power of the variable's norm in the case of adaptive methods). In our cases, the approximate $p$th order derivative is  calculated using either calls  to the $p$th order exact tensor oracle or its finite difference approximation at selected iterations  or by using information from the $(p-1)$-order derivative to form an inexact approximation of the  tensor.}
{For the sake of simplification and following standard notation in nonlinear optimization, we denote the gradient  as $g_k \eqdef \nabla_x^1 f(x_k)$. }
Our proposed algorithm follows the outline  of existing \texttt{ARp} regularization methods \cite{BirgGardMartSantToin17,Cartis2022-wb},
with the significant difference that the objective function $f(x_k)$ is never computed as done in \cite{OFFO-ARp,GraJerToin24}, and therefore
that the ratio of achieved to predicted reduction  is not used
to accept or reject a potential new iterate and to update the regularization parameter. 

{As our algorithm also uses  secant equations in order to update the approximate $p$th order tensor as in \cite{Welzel2024}, we provide here some context and notations for clarity.
	Following the terminology used in \cite{Welzel2024}, we introduce the averaged $p$th-order tensor between two iterates $x_k$ and $x_{k+1} = x_k + s_k$ where $s_k$ is the step, namely,}
{\begin{equation}\label{averagedtensor}
		\widetilde{T}_{k} \eqdef \int_{0}^{1} \nabla_x^p f(x_k+ts_k)\, dt.
	\end{equation}
	From the secant equation in \cite{Welzel2024}, we have that
	\begin{equation}\label{highordsecant}
		\widetilde{T}_{k}[s_k] = \nabla_x^{p-1} f(x_{k+1}) - \nabla_x^{p-1} f(x_k).
	\end{equation}
}
{Note that when $p=2$, we will denote $T_k$ and $\widetilde{T}_k$ as $B_k$ and $\widetilde{B}_k$ as is standard in quasi-Newton methods \cite{MoreJorge77}. After describing related work, we now present our Algorithm~\ref{AlgolazyOFFO} on page \pageref{AlgolazyOFFO}.
}

{At each iteration, we compute  the first $p-1$ derivatives. For the $p$th-order information, we use a generic subroutine  \texttt{Tensor\_p} (to be specified later on, in Section~\ref{Propupdates}) that takes various inputs such as the last estimated tensor $T_{k-1}$, the iteration count $k$, and the memory parameter $m$. In addition to all these inputs, \texttt{Tensor\_p} may also query exact $p$th-order derivative oracle or the $p-1$th one. Precise definitions of the $\texttt{Tensor\_p}$ routine will be presented after discussing the main features of Algorithm~\ref{AlgolazyOFFO}.  
	After forming the model $m_k$, we compute a step $s_k$ in order to approximately minimize this model. 
	Note that both tests \eqref{modelconditions} and \eqref{addedcond} follow \cite{GratToin21} and extend the more usual conditions where the step $s_k$ is chosen as 
	\[
	\|\nabla_s^1 m_k(s_k)\| \leq \theta_1 \|s_k\|^p,  \quad \textrm{ and } \max[0,-\lambda_{\min}(\nabla_s^2 m_k(s_k))] \leq \theta_2 \|s_k\|^{p-1}.
	\]
	Indeed, it is easy to verify that both \eqref{modelconditions} and \eqref{addedcond} hold at a local minimizer of $m_k$ with $\theta_1 , \, \theta_2 > 1$.
	After computing the step, we always accept it and employ the the update rule in \eqref{sigmakupdate}. By construction, $\sigma_k$ is a non-decreasing sequence. In the exact setting, this update rule is known to produce increasingly conservative regularization parameters, as shown in \cite{OFFO-ARp}. Nevertheless, it has also been shown to effectively handle inexact derivative information \cite{GraJerToin24}. In our setting, the inexactness is due to approximations of the $p$th-order derivative, since our aim is to develop a quasi-Newton approximation $(p=2)$ and utilize HOSU ($p \geq 3$) approximations studied in \cite{Welzel2024} and it is handled using the OFFO mechanism through the update rule \eqref{sigmakupdate}. More practical choices of $\sigma_k$ will be presented  in the numerical results section; see  Section~\ref{subseclazyHOSUnoisy}.
}

\begin{algorithm}
	\caption{Adaptive $p$th-order regularization with $p$th   Tensor approximation: \texttt{ARp-approx}}\label{AlgolazyOFFO}	
	\begin{algorithmic}[1]
		\Require An initial point $x_0 \in \Ren$, $\sigma_0 > 0$ are given, as well as the parameters
		$\theta_1> 1$,  {$\theta_2 > 1$}, a memory length parameter $m$ and {accuracy thresholds $(\epsilon_1, \epsilon_2)$.} Define also $\| s_{-1} \| = \cdots = {\|s_{-m}\|} \eqdef 1$. \quad {Set initial average tensor $\widetilde{T}_{-1} = 0$ and $T_{-1} = 0$.}
		\State $k \gets 0$.
		\While{{$\|\nabla_x^1 f(x_k)\| \geq \epsilon_1$ or $\lambda_{\min}(\nabla_x^2 f(x_k)) \geq -\epsilon_2$}}
		\State Compute the true derivatives $\nabla^i_x f(x_k)$ for $i \in \iibe{1}{p-1}$.
		
		\State Compute the $p$th-order tensor {$T_k$ such that} 
		\begin{equation}\label{$p$th  approx}
			T_k = {\texttt{Tensor\_p}(k,m, {T}_{k-1},  \widetilde{{T}}_{k-1}, \nabla_x^{p-1} f(x_k), \nabla_x^{p-1} f(x_{k-1}), \{\|s_{k-i}\|\}_{i=1}^m ).}
		\end{equation}

		%\If{$k \bmod m=0$}
		%Compute $T_k$ such that
		%\begin{equation}\label{Tkmodmcomp}
		%	\|T_k - \nabla_k^p f(x_k)\| \leq \min( \kappa_A \sum_{i=1}^{m} \|s_{k-i}\|, \, \kappa_B \mathds{1}_{p \geq 3} )
		%\end{equation}
		%\ElsIf{$k \bmod m \neq 0$}
		%   \If{$\kappa(W_k) \leq \kappa_C$}
		%  \State Update $T_{k}$ such that 
		% \begin{equation}\label{updatecondition}
			%	\|T_{k} - \widetilde{T}_{k-1}\|_{W_k} \leq \|T_{k-1} - \widetilde{T}_{k-1}\|_{W_k} 
			%\end{equation}
			%\Else
			
			%  \State Set 
			%  \begin{equation}\label{improveglobnorm}
				%  \|T_{k} - \widetilde{T}_{k-1}\| \leq \|T_{k-1} - \widetilde{T}_{k-1}\|
				%  \end{equation}
			%\EndIf
			%\EndIf
			
			\State Compute a step $s_k$ such that 
			\begin{align}\label{modelconditions}
				m_k(s_k) - m_k(0) \leq 0, \, \quad \quad \|\nabla_s^1\overline{ T_{f,p}}(s_k,s_k)\| \leq \frac{\theta_1 \sigma_k \|s_k\|^{p}}{p!}, \\   
				{\max[0, -\lambda_{\min}( \nabla_s^2\overline{ T_{f,p}}(x_k,s_k))] \leq \frac{\theta_2 \sigma_k \|s_k\|^{p-1}}{(p-1)!}}. \label{addedcond}
			\end{align}
			\State Set $x_{k+1} \gets x_k + s_k$, 
			\begin{equation}\label{sigmakupdate}
				\sigma_{k+1} \gets \sigma_k + \sigma_k \|s_k\|^{p+1}.
			\end{equation}
			\State $k \gets k+1$
			
			\EndWhile
		\end{algorithmic}
	\end{algorithm}

	\noindent
	Since various rates of convergence can be obtained depending on the quality of the approximation $T_k$, we now detail our condition on the \texttt{Tensor\_p} subroutine in order to obtain the optimal complexity of $p$th-order adaptive regularization methods \cite{BirgGardMartSantToin17}.

\newpage

\begin{cond}{Conditionpth}
The \texttt{Tensor\_p} outputs an approximation $T_k$ such that, 
if $k \bmod m=0$, the computed $T_k$ satisfies
\begin{equation}\label{Tkmodmcomp}
	\|T_k - \nabla_x^p f(x_k)\| \leq \min( \kappa_A \sum_{i=1}^{m} \|s_{k-i}\|, \, \kappa_B  ).
\end{equation}
Else ($k \bmod m \neq 0$)
\begin{equation}\label{updatecondition}
		\|T_{k} - \widetilde{T}_{k-1}\| \leq \kappa_C \|T_{k-1} - \widetilde{T}_{k-1}\|.
\end{equation}

\end{cond}
	We  prove later, in Section~\ref{analysis}, that using a $p$th-order approximation satisfying  Condition~\ref{Conditionpth} and the new update rule used in \eqref{sigmakupdate}, we  recover the optimal complexity obtained for the standard adaptive regularization obtained in \cite{BirgGardMartSantToin17}. {Naturally, the constants appearing in Condition~\ref{Conditionpth}, namely $(\kappa_A, \kappa_B, \kappa_C)$, depend on both the choice of parameters in the \texttt{Tensor\_p} subroutine and the geometry of the problem. However, no prior knowledge of these constants is required. Indeed, for each implementation of \texttt{Tensor\_p} described in Section~\ref{Propupdates}, the tensor approximation is constructed solely from oracle evaluations, the history of the derivative information collected along the iterates, and a set of user-defined parameters.} 

\noindent
Note that $\texttt{Tensor\_p}$ can be decomposed into two parts: a restart mechanism where one updates the approximation of the $p$th-order tensor and uses past step information to compute the $p$th-order approximation in \eqref{Tkmodmcomp}, and for the {remaining} iterations, the goal is to {calculate} a loose approximation of the average tensor $\widetilde{T}_{k-1}$ \eqref{updatecondition}.  
\noindent
It should be noted that the present study has so far focused exclusively on providing conditions on $T_k$ and has not yet provided a comprehensive algorithm.  However, as discussed in the introduction, we will cover a wide range of variants, {with} particular focus on the quasi-Newton method and High-Order Secant Updates (HOSU), where we  provide global convergence rates for this type of methods when applied to non-convex problems. 

%\begin{itemize}
%	\item An extension of the lazy Hessian paradigm \cite{DoikChayJag23,doikov2023zerothorder} to the high-order algorithms. In our case, we will recompute the $p$th   order tensor exactly every $m$ iterations as to satisfy \eqref{Tkmodmcomp}.
%	\item A new finite difference framework, where we use the $(p-1)th$ order derivative to compute an approximation of the $p$th   order tensor. Thanks also to the bound \eqref{Tkmodmcomp}, the stepsize of the finite difference is explicit and depends only past iterates information and algorithm hyper-parameters.
%	\item An inclusion of the new High-Order Secant Update (HOSU) recently proposed by \cite{Welzel2024} that extend the well-known quasi Newton updates of PSB \cite{Broyden1965} and DFP update rule \cite{Fletcher1963,Davidon59} for  $p\geq3$. Note that this update rule will be used for $(k \mod m \neq 0)$ \eqref{updatecondition} as this method will improve the estimate of $\widetilde{T}_{k-1}$ by using the secant property \eqref{highordsecant}.    
%\end{itemize}

We emphasize that we do not require an analytical expression of the averaged tensor $\widetilde{T}_{k-1}$ and just define it for ease of presentation, a standard practice in quasi-Newton methods \cite{Rodomanov2021,JinJiang24,JinAryan22}. For some variants of \texttt{ARp-approx}, we may need the action of $\widetilde{T}_{k-1}$ on $s_{k-1}$, where we have a closed analytical expression of the latter \eqref{highordsecant}.

In the following subsection, we  highlight how Condition~\ref{Conditionpth}  relates to the inexact derivative assumption introduced in \cite{GraJerToin24}. Since we use the same update rule of $\sigma_k$ as the latter, satisfying the assumption used in their analysis will allow us to use some of their derived results.

\subsection{Tensor Error Bound Analysis} \label{subsec-errbound}

After stating the generic algorithm {and analyzing its main features}, we now mention an additional set of assumptions that will be used during the rest of the paper.

\begin{assumption}\label{assum2}
	There exists a constant $f_{\rm low}$ such that $f(x) \geq f_{\rm low}$ for all $x \in \Ren$.
\end{assumption}
\begin{assumption}\label{assum3}
	The $p$th   tensor derivative of $f$ is globally Lipschitz continuous, that there exists, $L_p \geq 3$ such that
	\begin{equation}\label{Lip$p$th  }
		\| \nabla^p_x f(x) - \nabla^p_x f(y) \| \leq L_p \|x-y\|, \quad \text{ for all } x,y \in \Ren.
	\end{equation}
\end{assumption}
%Note that\eqref{Lipf} also implies  that 
%\begin{equation}\label{graderrorlip}
%	\|\nabla^1 f(x+s) - \nabla^1 f(x) - \nabla^2 f(x)s\| \leq \frac{L}{2} \|s\|^2
%\end{equation}
\begin{assumption}\label{assum4}
	{If $p\geq 4$}, there exists a constant $\kappa_{high} \geq 0$ such that
	\begin{equation}\label{kappahighbound}
		\min_{\|d\| \leq 1} \nabla^i_x f(x)[d]^i \geq -\kappa_{high} \quad \text{ for all } x \in \Ren \, \, \text{ and } i \in \iibe{3}{p-1},
	\end{equation}
\end{assumption}

\begin{assumption}\label{assum5}
	If $p \geq 3$, {the} $p$th   tensor derivative of $f$ is bounded
	\begin{equation}\label{bound$p$th  }
		\|\nabla^p_x f(x)\| \leq \kappa_p, \quad \text{ for all } x \in \Ren.
	\end{equation}
\end{assumption}

The set of Assumptions 1--3 are standard when studying $p$th-order adaptive methods \cite{BirgGardMartSantToin17}. Assumption~\ref{assum4} is standard when studying OFFO adaptive regularization methods \cite{OFFO-ARp, GraJerToin24} with a slight improvement that in our case we only consider a bound for $i \in \iibe{3}{p-1}$. {Observe that Lipschitz continuity of derivatives up to order $p - 1$ covers Assumption~\ref{assum4}. However, the  assumption used here is less restrictive for even-order tensor as it only requires a  lower-bound on the contraction of the tensor along all directions. }    Note that Assumption~\ref{assum5} is a stronger requirement than the condition imposed on the $p$th   tensor derivative in \cite{OFFO-ARp,GraJerToin24} when $p$ is even and strictly larger than two \footnote{{$f(x) = \frac{1}{N} \sum_{i=1}^{N} \phi(a_i^\intercal x - b_i) \,\text{ where } \,\phi(\theta) = \frac{\theta^2}{1+\theta^2},$ with $a_i \in \mathbb{R}^n$, $b_i \in \mathbb{R}$, and $N$ denotes the number of samples is an illustrative example that satisfy all the stated assumptions for all $p \geq 2$. It is infinitely differentiable and all derivatives of $f$ are bounded.}}

% correct{As an illustrative example that satisfies both Assumption~\ref{assum4} and Assumption~\ref{assum5}, we propose a robust linear regression with the smooth biweight loss \cite{Beaton74},
	%Here,  This objective function satisfies all the assumptions required for $p\geq3$,  because the derivatives of any order of the univariate function $\phi$ are themselves bounded.
	%	}

Let us now detail the impact of Condition~\ref{Conditionpth} on the approximation of $\nabla_x^p f(x_k)$. { But before, remark that
	from \eqref{averagedtensor}, Assumption~\ref{assum3} and Assumption~\ref{assum5}, we easily obtain the following error bound
	\begin{equation}\label{avregaredtotrue}
		\| \widetilde{T}_{k} - \nabla_x^p f(x_k) \| \leq \frac{L_p}{2} \|s_k\|, \quad 
		\| \widetilde{T}_{k} - \nabla_x^p f(x_{k+1}) \| \leq \frac{L_p}{2} \|s_k\|,
	\end{equation}
	and also,
	\begin{equation}\label{averagedbound}
		\|\widetilde{T}_{k}\| \leq \kappa_p.
\end{equation}}

{We now prove an error bound between $T_k$ and $\nabla_x^p f(x_k)$, which combines the two previous inequalities with Condition~\ref{Conditionpth}.}
\begin{lem}{pthbound}
Let $k \geq 0$. Suppose that Assumption~\ref{assum1} and Assumption~\ref{assum3} hold. Then, 
\begin{equation}\label{condTkapprox}
	\|T_k - \nabla_x^p f(x_k)\|^{p+1} \leq \kappa_D \sum_{i=\max[-m,k-2m+1]}^{k-1} \|s_i\|^{p+1}
\end{equation}
where 
\begin{equation}\label{kappaDdef}
	\kappa_D \eqdef (2m-1)^p \max[\kappa_C^{(m-1)} \kappa_A, \frac{L_p}{2} (1 + \kappa_C) \kappa_C^{(m-1)}]. 
\end{equation}
Suppose now that Assumption~\ref{assum1} and Assumption~\ref{assum5} hold. Then
\begin{equation}\label{Tkbound}
	\|T_k\| \leq \kappa_E \eqdef \kappa_C^{(k-k_m)} (\kappa_B + \kappa_p) + (1+\kappa_C) \kappa_p \sum_{j=0}^{m-1} \kappa_C^{j}.
\end{equation}
\end{lem}

\begin{proof}
		First note that  \eqref{condTkapprox} is true for $k=0$ since \eqref{Tkmodmcomp} is used to compute $T_0$ and that $\kappa_C \geq 1$. 
	Let $k \geq 1$ and define $k_m = k - (k \mod m)$ and consider $j \in \iibe{k_m + 1}{k}$. 
	From \eqref{updatecondition}, triangular inequality,  and using \eqref{avregaredtotrue} twice, we have that
	\begin{align*}
		\|T_j - \nabla_x^p f(x_j)\| &\leq \|T_{j}- \widetilde{T}_{j-1}\| + \|\widetilde{T}_{j-1} - \nabla_x^p f(x_j)\| \\
		&\leq \kappa_C\|T_{j-1} - \widetilde{T}_{j-1}\| + \frac{L_p}{2} \|s_{j-1}\|\\
		&\leq \kappa_C\|T_{j-1} - \nabla^p_x f(x_{j-1}) \| + \kappa_C\| \nabla^p_x f(x_{j-1})  - \widetilde{T}_{j-1}\| + \frac{L_p}{2} \|s_{j-1}\| \\
		&\leq \kappa_C \|T_{j-1} - \nabla^p_x f(x_{j-1}) \| + \frac{L_p}{2} (1 + \kappa_C) \|s_{j-1}\|
	\end{align*}
	Multiplying each inequality by {$\kappa_C^{k-j}$} and summing them for $j \in \iibe{k_m + 1}{k}$, using the first part of \eqref{Tkmodmcomp} that $k-k_m \leq (m-1)$, we derive that
	\begin{align*}
		\|T_k- \nabla^p_x f(x_k)\| &\leq {\kappa_C^{k-k_m}} \|T_{k_m} - \nabla^p_x f(x_{k_m})\| + \frac{L_p}{2} (1 + \kappa_C) \sum_{j= k_m + 1}^k {\kappa_C^{k-j}} \|s_{j-1}\| \\
		&\leq  {\kappa_C^{m-1}} \kappa_A \sum_{i=1}^{m} \|s_{k_m - i}\| + \frac{L_p}{2} (1 + \kappa_C) \sum_{j= k_m + 1}^k {\kappa_C^{k-j}} \|s_{j-1}\| \\
		&\leq {\kappa_C^{m-1}} \kappa_A {\sum_{i=1}^{m}} \|s_{k_m - i}\| + \frac{L_p}{2} (1 + \kappa_C) {\kappa_C^{m-1}} \sum_{j= k_m + 1}^k  \|s_{j-1}\| \\
		&\leq \max[{\kappa_C^{m-1}} \kappa_A, \frac{L_p}{2} ( 1 + \kappa_C^p) {\kappa_C^{m-1}}] \sum_{i=\max[-m,k-2m+1]}^{k-1} \|s_i\| \\
	\end{align*}
	where, in the last inequality, we have at  most $(2m-1)$ term. Taking the last inequality to the power $(p+1)$ and using that, by convexity $(a_1 + a_2 + \dots + a_{2m-1})^{p+1} \leq (2m-1)^p (a_1^{p+1} + a_2^{p+1} + \dots + a_{2m-1}^{p+1} )$ yields \eqref{condTkapprox}.
	
	We have now proved the first part of the Lemma. We now turn to the proof of \eqref{Tkbound} for $p \geq 3$.
	Let  $j \in \iibe{k_m + 1}{k}$. Rearranging now \eqref{updatecondition} and  the fact that \eqref{bound$p$th  } holds for $p \geq 3$ yields that
	\[
	\|T_j\| \leq \kappa_C \|T_{j-1}\| + (1+\kappa_C) \|\widetilde{T}_{j-1}\| \leq \kappa_C\|T_{j-1}\| + (1+\kappa_C) \kappa_p.
	\]
	Multiplying the last inequality by ${\kappa_C^{k-j}}$, summing for $j \in \iibe{k_m +1}{k}$ and {using} $\|T_{k_m}\|  \leq \kappa_B + \kappa_p$ from the second term of \eqref{Tkmodmcomp}, we derive that
	\begin{align*}
		\|T_k\| &\leq {\kappa_C^{k-k_m}} \|T_{k_m}\| + (1+\kappa_C) \kappa_p \sum_{j=k_m + 1}^{k} {\kappa_C^{k-j}} \\
		&\leq {\kappa_C^{k-k_m}} (\kappa_B + \kappa_p) + (1+\kappa_C) \kappa_p \sum_{j=0}^{m-1} \kappa_C^{j},
	\end{align*}
	using that $k-k_m \leq m$ proves the last item of the Lemma.
\end{proof}

	We now compare the  inexact condition  \eqref{condTkapprox} on the $p$th-order tensor with the analysis developed in \cite{GraJerToin24}.
\noindent
We define 
\begin{equation}\label{xikdef}
	\xi_k \eqdef \sum_{i=1}^{2m-1} \|s_{k-i}\|^{p+1}
\end{equation}
with the convention that
\begin{equation}\label{sigmajnegdef}
	\| s_{-1} \| = \cdots = \|s_{-2m+1}\| \eqdef 1 \text{ and }\sigma_{j} = \frac{\sigma_{0}}{2^{-j}}, \, \quad j \in \iibe{-2m+1}{-1}.
\end{equation}
From \eqref{condTkapprox} and \eqref{xikdef}, we derive that
\begin{equation}\label{Tkxikrelation}
	\|T_k - \nabla_x^p f(x_k)\|^{p+1} \leq \kappa_D \xi_k.
\end{equation}

Thus, we deterministically satisfy the probabilistic conditions given in \cite[AS.5]{GraJerToin24} with $\kappa_D$ defined in \eqref{kappaDdef}. And so all the results related to Lipschitz continuity and tensor error bounds \cite{GraJerToin24} paper hold. In particular, the intermediate results \cite[{Lemma~6, Lemma~7, Lemma~9}]{GraJerToin24} hold.

We now discuss where we differ in our paper from the analysis \cite{GraJerToin24} and explain why these differences are important when working with adaptive methods.
Note that from Assumption~\ref{assum4} and Assumption~\ref{assum5}, we have removed all conditions on the negative curvature of the second-order estimate, which was assumed to be bounded in \cite{GraJerToin24}. 
To remove these conditions, we introduced the second-order stationarity condition \eqref{addedcond} when computing the step. Intuitively, this condition controls the length of the step size in relation to the approximate negative curvature. 
It is also reasonable to remove all assumptions on the Hessian and second-order estimates, since 
only Assumptions~\ref{assum1}--\ref{assum3} are used in the analysis of the cubic regularization method \cite{NesPolyak06,CartisGT11a}. For the sake of simplicity and brevity, we denote 
\begin{equation}\label{chikdef}
	\chi_k = \left\{ 	\begin{aligned}
		&\max[0, -\lambda_{\min}(\nabla_x^2 f(x_k))] &&\text{ if } p \geq 3, \\
		& \max[0, -\lambda_{\min}(B_k)] &&\text{ if } p=2.
	\end{aligned}\right.
\end{equation}

\noindent
and true negative curvature as 
\begin{equation}\label{trueneg}
	\beta_k = \max[0,-\lambda_{\min}(\nabla_x^2 f(x_k))].
\end{equation}

In subsequent developments, we only provide additional results that focus on the second-order measures of criticality ($\beta_k$, $\chi_k$), since results on $\|g_k\|$ have been established in \cite{GraJerToin24} . In detail, we  establish new bounds on the negative curvature and establish a relation between the stepsize and  negative curvature. Afterwards, building on the work of \cite{GraJerToin24} and using the update rule of $\sigma_k$ \eqref{sigmakupdate}, we derive the complexity bounds accordingly. 
 
 \section{Complexity Analysis of \texttt{ARp-approx} }\label{analysis}
In this section, we provide a detailed analysis of the algorithmic complexity of Algorithm~\ref{AlgolazyOFFO}. We  show that it achieves the same rates as standard adaptive methods \cite{Cartis2022-wb} to reach second-order approximate stationary point. As first-order stationarity has been already dealt with in \cite{GraJerToin24}, we reuse the established results accordingly. Therefore, we will only derive new bounds on the negative curvature $\chi_k$. Proofs similar to those presented in \cite{GraJerToin24} are deferred to the Appendix for brevity.   

We start by proving an error bound between the Hessian of the objective function $f$ at iterate $k+1$ and the Hessian of the inexact Taylor approximation $\overline{ T_{f,p}}$.
\begin{lem}{boundliphess}
Suppose that Assumption~\ref{assum1} and Assumption~\ref{assum3} hold. Then, we have that 
\begin{equation}\label{Hessbounderror}
\frac{\|\nabla_x^2 f(x_{k+1}) -  \nabla_s^2\overline{ T_{f,p}}(x_k,s_k)\|^\sfrac{p+1}{p-1}}{\sigma_{k+1}^\alpha} \leq \kappa_1 \frac{\|s_k\|^{p+1}}{\sigma_{k+1}^\alpha} + \kappa_2 \frac{\xi_k}{\sigma_k^\alpha},
\end{equation}

where 
\begin{equation}\label{kappa1andtwodef}
	\kappa_1 \eqdef 3^\sfrac{2}{p-1} \left( (\frac{L_p}{(p-1)!})^\sfrac{p+1}{p-1} + (\frac{p-2}{(p-1)!})^\sfrac{p+1}{p-1}\right), \, \, \kappa_2 \eqdef 3^\sfrac{2}{p-1} \kappa_D \left(\frac{1}{(p-1)!}\right)^\sfrac{p+1}{p-1},
\end{equation}
and $\xi_k$ is defined in \eqref{xikdef}.
\end{lem}

\begin{proof}
Using  triangle inequality, Hessian error bound for functions with Lipschitz $p$th order tensor \cite[Corollary~A.8.4]{Cartis2022-wb} and the fact that the true and approximate Taylor expansions \eqref{taylor}, \eqref{approxtaylor} differ only for the $p$th tensor, we derive,  

\begin{align*}
\|\nabla_x^2 f(x_{k+1}) - \nabla_s^2 \overline{ T_{f,p}}(x_k,s_k)\| &\leq \|\nabla_x^2 f(x_{k+1}) -  {\nabla_s^2 T_{f,p}}(x_k,s_k)\| \\
&+ \|\nabla_s^2 T_{f,p}(x_k,s_k) -  \nabla_s^2 \overline{ T_{f,p}}(x_k,s_k)\| \\
&\leq \frac{L_p}{(p-1)!} \|s_k\|^{p-1} + \frac{\|T_k - \nabla_x^p f(x_k)\| \|s_k\|^{p-2}}{(p-2)!}.
\end{align*}
Using Young's inequality with $p = \frac{p-1}{p-2}$ and $q = p-1$ for the last term in the above, we derive that
\begin{align*}
\|\nabla_x^2 f(x_{k+1}) -  \nabla_s^2\overline{ T_{f,p}}(x_k,s_k)\| &\leq \frac{L_p}{(p-1)!} \|s_k\|^{p-1} +  \frac{\|T_k - \nabla_x^p f(x_k)\|^{p-1}} {(p-1)!} \\ 
&+ \frac{p-2}{(p-1)!} \|s_k\|^{p-1}.
\end{align*}
Taking the last inequality to the power $\frac{p+1}{p-1}$, and using the convexity of $x^\sfrac{p+1}{p-1}$, we obtain that
\begin{align*}
\|\nabla_x^2 f(x_{k+1}) -  \nabla_s^2\overline{ T_{f,p}}(x_k,s_k)\|^\sfrac{p+1}{p-1} &\leq 3^\sfrac{2}{p-1} \Biggl(  (\frac{L_p}{(p-1)!})^\sfrac{p+1}{p-1} \|s_k\|^{p+1} \\ 
&+  \frac{\|T_k - \nabla_x^p f(x_k)\|^{p+1}}{(p-1)!^\sfrac{p+1}{p-1}} +  (\frac{p-2}{(p-1)!})^\sfrac{p+1}{p-1} \|s_k\|^{p+1} \Biggr)
\end{align*}
Dividing by $\sigma_{k+1}^{\alpha}$, using that $\sigma_{k+1}\geq \sigma_{k}$, and  \eqref{Tkxikrelation}, we obtain the desired result with the defined \eqref{kappa1andtwodef}. 
\end{proof}

{The next set of lemmas follow the same proof techniques as developed in \cite{GraJerToin24} for the stochastic setting, but  here, they are adapted to the inexact derivatives' case and extended to establish convergence to a second-order stationarity point. Since the proofs are similar, they are deferred to the appendix; but they are still included  for the sake of completeness. }
The next lemma provides a bound on the  negative curvature of the function $f$ at iteration $k+1$ with the respect to the regularization parameter of iterations $j \in \iibe{k-m}{k+1}$.

\begin{lem}{boundbetak}
	Suppose that Assumption~\ref{assum1} and Assumption~\ref{assum3} hold. Then, we derive that
	\begin{equation}\label{crucialhess}
\frac{\beta_{k+1}^\sfrac{p+1}{p-1}}{\sigma_{k+1}^\sfrac{p+1}{p-1}} \leq \kappa_3 \frac{\sigma_{k+1}-\sigma_{k}}{\sigma_{k+1}} + \kappa_4 \sum_{j=k-2m+1}^{k-1} \frac{\sigma_{j+1}-\sigma_{j}}{\sigma_{j+1}},
	\end{equation} 
	
		\begin{equation}\label{crucialhesssecond}
		\frac{\beta_{k+1}^\sfrac{p+1}{p-1}}{\sigma_{k+1}^\sfrac{2}{p-1}} \leq \kappa_3 (\sigma_{k+1}-\sigma_{k}) + \kappa_4 \sum_{j=k-2m+1}^{k-1} (\sigma_{j+1}-\sigma_{j}),
	\end{equation} 
	
	\begin{equation}\label{kappa34def}
	\kappa_3 \eqdef \frac{\twoopptwo}{\sigma_{0}^\sfrac{p+1}{p-1}} \left( \kappa_1 + \frac{(\theta_2 \sigma_{0})^\sfrac{p+1}{p-1}}{(p-1)!^\sfrac{p+1}{p-1}}\right) \quad \kappa_4 \eqdef \frac{2^\sfrac{(2m-1)(p-1)+2}{p-1}}{\powppminsone{\sigma_{0}}} \kappa_2,
	\end{equation}
where $\kappa_1$ and $\kappa_2$ are defined in \eqref{kappa1andtwodef} and $\beta_k$ defined in \eqref{trueneg}.
\end{lem}

\begin{proof}
	For the sake of brevity and for similarity with \cite{GraJerToin24}, the proof is deferred to the Appendix~\ref{firstappendix}. 

\end{proof}

The next lemma gives an  upper-bound on the step length. Similar results have been stated in the adaptive regularization literature \cite{OFFO-ARp,GraJerToin24,CartGoulToin19}, but this one will specifically showcase the impact of the negative curvature of the approximate Hessian $\chi_k$ \eqref{chikdef}. 

\begin{lem}{skbound}
	Suppose that Assumption~\ref{assum1}, Assumption~\ref{assum4} and Assumption~\ref{assum5} hold. Then, we have that
	\begin{equation}\label{skboundneg}
		\|s_k\| \leq 2 \max \left[\eta, \, \left(\frac{(p+1)!\|g_k\|}{\sigma_{k}}\right)^\sfrac{1}{p}, \, \, \left(\frac{(p+1)! \chi_k}{2 \sigma_{k}} \right)^\sfrac{1}{p-1} \right]
	\end{equation}
	where
	\begin{equation}\label{etadef}
		\eta \eqdef \max\Biggl[ \, \max_{i \in \iibe{3}{p-1} } \left(\frac{\kappa_{high} (p+1)!}{i! \sigma_{0}}\right)^\sfrac{1}{p-i+1} , \frac{\kappa_E\indica{p\geq3}}{\sigma_{0}}  \Biggr],
	\end{equation}
	and $\kappa_E$ is defined in \eqref{Tkbound}.
	Moreover
	\begin{equation}\label{skboundizycase}
		\|s_k\|^{p+1} \mathds{1}_{\|s_k\| \leq 2\eta} \leq (1+(2\eta)^{p+1}) \frac{\sigma_{k+1} - \sigma_{k}}{\sigma_{k+1}} \mathds{1}_{\|s_k\| \leq 2\eta}.
	\end{equation}
\end{lem}
\begin{proof}
	Since the proof is similar to the one  in \cite[{Lemma~8}]{GraJerToin24}, we defer it to Appendix~\ref{firstappendix}.
\end{proof} %epr

We now provide a Lemma that bounds $\left(\frac{\chi_k^\sfrac{p+1}{p-1}}{\powppminsone{\sigma_{k}}}\right)$ as it will allow us to derive then an upper-bound on $\|s_k\|^{p+1}$ from Lemma~\ref{skbound}.

\begin{lem}{chiklemma}
Suppose that Assumption~\ref{assum1}, Assumption~\ref{assum3}, Assumption~\ref{assum4} and Assumption~\ref{assum5} hold. and let $k\geq1$. Then, we have that 
\begin{align}\label{xikbound}
	\frac{\chi_k^\sfrac{p+1}{p-1}}{\sigma_{k}^\sfrac{p+1}{p-1}} &\leq  \frac{2^{2m+1} \kappa_D\indica{p=2}}{\sigma_{0}^3} \sum_{j=k-2m+1}^{k-1} \frac{\sigma_{j+1}-\sigma_{j}}{\sigma_{j+1}} + 
	4 \kappa_3 \frac{\sigma_{k}-\sigma_{k-1}}{\sigma_{k}}. \nonumber \\ 
	&+ 4 \kappa_4 \sum_{j=k-2m}^{k-2} \frac{\sigma_{j+1}-\sigma_{j}}{\sigma_{j+1}} 
\end{align}
And for $k=0$, we have that
\begin{equation}\label{xikzero}
	\frac{\chi_0^\sfrac{p+1}{p-1}}{\sigma_{0}^\sfrac{p+1}{p-1}} \leq 4 \left( \frac{\max[0,-\lambda_{\min}(\nabla_x^2 f(x_0))]^\sfrac{p+1}{p-1}}{\sigma_{0}^\sfrac{p+1}{p-1}} + \frac{m^3 \kappa_A^3 \mathds{1}_{p=2}}{\sigma_{0}^3} \right).
\end{equation}
\end{lem}

\begin{proof}
The proof argument is analogous to the one in \cite{GraJerToin24}, and so it is deferred to Appendix~\ref{thirdproof}.
\end{proof}

Combining now the results of Lemma~\ref{skbound} and Lemma~\ref{chiklemma}, we are  ready to state a Lemma that upper-bounds $\sum_{j=0}^k \|s_j\|^{k+1}$. Since the technique of the proof overlap with some elements already developed in \cite{GraJerToin24}, we will refer to that paper when needed. Before that we need the following technical result that bounds a specific type of sum.

\begin{lem}{ajsum}
Let $\{(a_j)_{j \in \iibe{m}{n}}\}$ be a positive,
nondecreasing sequence with $m < n$ and $(m,\,n) \in \mathbb{Z}^2$. Then, we have that 
\begin{equation}\label{logsumbound}
\sum_{j=m+1}^n \frac{a_j - a_{j-1}}{a_j} \leq \log\left( a_n \right) - \log \left( a_{m}\right). 
\end{equation}
\end{lem}

\begin{proof}
See \cite[Lemma~3.5]{GraJerToin24}.
\end{proof}

\begin{lem}{sumsjbound}
	Suppose that Assumptions~\ref{assum1}, \ref{assum3}--\ref{assum5} hold. Then, we have that
	\begin{equation}\label{exprsumsjbound}
		\sum_{j=0}^{k} \|s_j\|^{p+1} \leq \kappa_{const} + \kappa_{log} \log(\sigma_{k+1}),
	\end{equation} 
	where $\kappa_{const}$ and $\kappa_{log}$ are problem-dependent constants. 
\end{lem}

\begin{proof}
	{The proof is deferred to Appendix~\ref{fourthproof}.} 
\end{proof}

From the last Lemma~\ref{sumsjbound}, we follow  closely the analysis developped in \cite{GraJerToin24}. In particular, we can establish the following Lemma. 

\begin{lem}{sigmakbound}
	Suppose that Assumptions~\ref{assum1}--\ref{assum5} hold. Then, there exists $\sigma_{\max}$ such that 
	\begin{equation}\label{sigmakboundexpr}
		\sigma_{k} \leq \sigma_{\max}
	\end{equation}
	where $\sigma_{\max}$ depends on problem constants.
\end{lem}

\begin{proof}
	The proof follows closely the argument in \cite[{Lemma~11}]{GraJerToin24}; deferred to Appendix~\ref{secondappendix}. 
\end{proof}

For more discussions on the dependency of $\sigma_{\max}$ in \eqref{sigmamaxrxpr} and how to develop more explicit upper bounds, we refer the reader to the analysis provided in \cite{GraJerToin24}.

We state our new theorem on the complexity of {\texttt{ARp-approx}}. 
\begin{theo}{complexbound}
Suppose that Assumptions~\ref{assum1}--\ref{assum5} hold. Then Algorithm~\ref{AlgolazyOFFO} outputs a sequence of iterates such that
\begin{equation}\label{gkbound}
	\min_{j \in \iibe{1}{k+1}} \|g_j\| = \mathcal{O}\left(\frac{1}{(k+1)^\sfrac{p}{p+1}}\right).
\end{equation}
Moreover, we have that 
\begin{equation}\label{betakbound}
\min_{j \in \iibe{1}{k+1}} \max[0,-\lambda_{\min}(\nabla_x^2 f(x_{j}))] =	\min_{j \in \iibe{1}{k+1}} \beta_{j} = \mathcal{O}\left(\frac{1}{(k+1)^\sfrac{p-1}{p+1}}\right).
\end{equation}

\end{theo}

\begin{proof}
For the first part of the Lemma, see the proof  of \cite[Theorem~3.8]{GraJerToin24}. We will focus now on the second part. 
Let $k \geq 1$ and $j \in \iibe{0}{k}$. Using \eqref{crucialhesssecond}, simplifying the upper bound, using that $\sum_{j= 0}^k \sum_{\ell = j-2m+1}^{j-1} \sigma_{j+1}-\sigma_{j} \leq 2m \sigma_{k}$ from \eqref{sigmakupdate}, and  Lemma~\ref{sigmakbound}, we derive that
\begin{align}
	\sum_{j=0}^{k}\frac{\beta_{j+1}^\sfrac{p+1}{p-1}}{\sigma_{j+1}^\sfrac{2}{p-1}} &\leq 	\sum_{j=0}^{k} \kappa_3 (\sigma_{j+1}-\sigma_{j}) + \kappa_4 	\sum_{j=0}^{k} \sum_{\ell=j-2m+1}^{j-1} (\sigma_{\ell+1}-\sigma_{\ell}) \nonumber \\
	&\leq \kappa_3 \sigma_{k+1} + 2m \kappa_4 \sigma_{k} \leq (\kappa_3+ 2m \kappa_4) \sigma_{\max} \nonumber.
\end{align}
Lower-bounding the left-hand side of the previous inequality and using again Lemma~\ref{sigmakbound}, we obtain that
\[
\min_{j \in \iibe{0}{k}} \frac{\beta_{j+1}^\sfrac{p+1}{p-1}}{\sigma_{\max}^\sfrac{2}{p-1}} (k+1) \leq 	\sum_{j=0}^{k}\frac{\beta_{j+1}^\sfrac{p+1}{p-1}}{\sigma_{j+1}^\sfrac{2}{p-1}} \leq (\kappa_3+ 2m \kappa_4) \sigma_{\max}. 
\]
Further rearrangement of the last inequality yields \eqref{betakbound}. 
\end{proof}

Note that the formulation of Theorem~\ref{complexbound} implies that a $(\epsilon_1, \, \epsilon_2)$ second-order approximate stationary point can be reached in $\mathcal{O}\left(\max(\epsilon_1^{-(p+1)/p}, \epsilon_2^{-(p+1)/(p-1)} ) \right)$ iterations.
{Naturally, the $\mathcal{O}$ notation hides additional dependencies on problem-dependent constants, such as $L_p$ or $m$.
	Explicit expressions for these constants have already been derived in \cite{GraJerToin24}. The choice of several hyperparameters, including $m$ and $\sigma_0$, affects both these hidden constants and the practical performance of the algorithm (see \cite[Section~5]{GraJerToin24}).
	Nevertheless, the constants concealed by the $\mathcal{O}$ notation are sufficiently intricate that an optimal choice of $m$ cannot be inferred directly from the complexity bound (see \cite{GraJerToin24} for full details). Moreover, the initialization of Algorithm~\ref{AlgolazyOFFO} fixes the lengths of the negative  steps to one. Choosing a different value, or even a sequence of values, would modify the constants appearing in the complexity estimate and, consequently, influence the optimal choice of $m$.
	We also note that \cite{doikov2023zerothorder} proposes a particular choice of $m$ that alleviates the dependence of the complexity bound on the problem dimension $n$. However, the proof techniques developed therein differ substantially from those employed in the present work and cannot be directly transferred to our algorithmic framework.
}

Thus, we have retrieved the complexity of standard tensor methods \cite{BirgGardMartSantToin17} despite not using  function evaluations and using approximate tensors. Note that we could have obtained the same results by assuming knowledge of the Lipschitz constant, by setting the regularization parameter to a fixed value in advance, or by using an appropriate line search condition every $m$ steps as in \cite{DoikChayJag23}.  However, we preferred to use the OFFO approach because it has been shown to be robust even when errors occur in derivatives of order 1 to $(p-1)$; see \cite{OFFO-ARp,GraJerToin24} for further discussion and illustrations of this fact.

\noindent
We now state a corollary that focuses on  practical case of interest $p=2$.

\begin{corr}{caseptwo}
	Suppose that Assumptions~\ref{assum1}--\ref{assum3}. Then Algorithm~\ref{AlgolazyOFFO} for $p=2$ outputs a sequence of iterates $\{x_j\}_{j=0}^{k+1}$ such that
	\begin{align*}
			\min_{j \in \iibe{1}{k+1}} \|g_j\| &= \mathcal{O}\left(\frac{1}{(k+1)^{2/3}}\right)  \\
			 \min_{j \in \iibe{1}{k+1}} \max[0,-\lambda_{\min}(\nabla_x^2 f(x_j))] &= \mathcal{O}\left(\frac{1}{(k+1)^{1/3}}\right)
	\end{align*}
	
\end{corr}

It is important to note that, unlike the approaches described in \cite{GraJerToin24, OFFO-ARp}, no assumption is made about the boundness of the negative curvature, be it exact or inexact. This is due to the additional condition on the computed step at each iteration \eqref{addedcond} and a finer analysis when bounding the step in Lemma~\ref{skbound} and additional results (Lemma~\ref{boundliphess}, Lemma~\ref{boundbetak} and Lemma~\ref{chiklemma}).

Given the recent focus in the research community on practical variants of third-order tensor methods
\cite{Cartisetal24,CartisZhu25, BirginGardenghiMartSantos19}, and given that our algorithmic frameworks allow for third-order tensor approximation {involving} the Hessian, as it will be shown in Section~\ref{Propupdates}, we detail the result of Theorem~\ref{complexbound} for $p=3$.

\begin{corr}{case$p$three}
		Suppose that Assumptions~\ref{assum1}--\ref{assum5}. Then Algorithm~\ref{AlgolazyOFFO} for $p=3$ outputs a sequence of iterate $\{x_j\}_{j=0}^{k+1}$  such that
	\begin{align*}
		\min_{j \in \iibe{1}{k+1}} \|g_j\| &= \mathcal{O}\left(\frac{1}{(k+1)^{3/4}}\right)  \\
		\min_{j \in \iibe{1}{k+1}} \max[0,-\lambda_{\min}(\nabla_x^2 f(x_j))] &= \mathcal{O}\left(\frac{1}{(k+1)^{1/2}}\right)
	\end{align*}
\end{corr}

In the next section, we will discuss in detail possible updates that approximate the $p$th tensor such that \eqref{updatecondition} holds. In particular, we will explore three different types of approximation, the lazy variant where we compute the exact $p$th derivative every $m$ steps, using finite differences of the $(p-1)$th order to approximate the $p$th order every $m$ steps, and incorporating the high-order secant updates in \cite{Welzel2024} (which includes PSB and DFP updates when $p=2$) to improve the tensor estimate between each recalculation of the $p$th tensor. 

\section{Proposed Update Rules}\label{Propupdates}

In this section, we  outline the various algorithmic variants encompassed by our \texttt{ARp-approx} framework.  {In particular, we will give several examples of \texttt{Tensor\_p} subroutines  that satisfy the Condition~\ref{Conditionpth}, under which the complexity analysis developed in Section~\ref{analysis} holds.  }  Firstly, in Section~\ref{sub-lazy} we start with a lazy variant in which the exact $p$th-order tensor is evaluated exactly every $m$ iterations. Next, we introduce a finite difference variant with an explicit discretization step size formula (Section~\ref{subs-fd}). Finally, Section~\ref{subs-HOSU} discusses  the incorporation of secant updates.

\subsection{Lazy Tensor} \label{sub-lazy}
One simple way of satisfying \eqref{updatecondition} is to set $T_k = T_{k-1}$ for $(k \mod m \neq 0)$. For the sake of clarity, let us restate our new variant of $\texttt{ARp-approx}$ which will be denoted by $\texttt{ARp-Lazy}$.

\begin{algorithm}
	\caption{Lazy adaptive $p$th order regularization (\texttt{ARp-Lazy})}\label{ARpLazy}	
	\begin{algorithmic}
		\Require Same as Algorithm~\ref{AlgolazyOFFO} with $\kappa_C = 1$, $\kappa_A = \kappa_B = 0$.
		\State
		{Follow the steps of} Algorithm~\ref{AlgolazyOFFO}, where \texttt{Tensor\_p} is set as: 
		
		\If{$k \mod m =0$}
		    \begin{equation}\label{exactcomp}
		    	T_k = \nabla_x^p f(x_k)
		    \end{equation}
		 \Else
		 \begin{equation}\label{keepconstantoffo}
		 	T_{k} = T_{k-1}
		 \end{equation}
		\EndIf
	\end{algorithmic}
\end{algorithm}

\begin{lem}{proofArplazy}
	Algorithm~\ref{ARpLazy} is a valid instance of Algorithm~\ref{AlgolazyOFFO} with {$\kappa_C = 1$, $\kappa_A = \kappa_B = 0$ in Condition~\ref{Conditionpth}}.
\end{lem}
\begin{proof}
	Note that \eqref{Tkmodmcomp} is valid for $\kappa_A = \kappa_B = 0$ since we compute the exact $p$th tensor. Remark that since $T_k = T_{k-1}$ from \eqref{keepconstantoffo} for ($k\mod m \neq 0$), \eqref{updatecondition} holds as an equality with $\kappa_C = 1$.
\end{proof}

We remark that from Theorem~\ref{complexbound} and the mechanism of computing the tensor derivative in Algorithm~\ref{ARpLazy}, {the results of Theorem~\ref{complexbound} apply and we would require at most 
	\newline
	$\mathcal{O}\left(\max[\epsilon_1^{-(p+1)/p}, \, \, \epsilon_2^{-(p+1)/(p-1)}]\right)$ iterations to reach an $(\epsilon_1, \, \epsilon_2)$ second-order stationarity point.} 

\begin{theo}{lazyoffo}
Suppose that Assumptions 1--5 hold. Then Algorithm~\ref{ARpLazy} requires at most \\
 $\mathcal{O}\left(\max[\epsilon_1^{-(p+1)/p}, \, \, \epsilon_2^{-(p+1)/(p-1)}]\right)$ iterations and evaluations of $\{ \nabla_x^i f\}_{i=1}^p$ to produce a vector $x_\epsilon \in \mathbb{R}^n$ such that $\|\nabla_x^1 f(x_\epsilon)\| \leq \epsilon$ and $\lambda_{\min}(\nabla_x^2 f(x_\epsilon)) \geq -\epsilon_2$. 
\end{theo}

\begin{proof}
The result follows by combining both Theorem~\ref{complexbound} and Algorithm~\ref{ARpLazy}.
\end{proof}

We now discuss our method for $p=2$. In that case, one recover the lazy Hessian framework which has recently been studied in nonconvex settings \cite{DoikChayJag23}. An advantage  of our adaptive lazy Hessian algorithm is that it always accepts the computed step $s_k$ whereas the adaptive variant proposed in \cite{DoikChayJag23} may reject the last $m$ steps if they fail to yield sufficient decrease. One  disadvantage is that it still computes the exact $p$th order tensor from time to time, which may require costly and unavailable information. 

In the next subsection, we will detail how we can develop a finite difference variant of Algorithm~\ref{AlgolazyOFFO}, therefore bypassing the need to compute $p$th order tensor exactly. 

\subsection{Finite Difference Variant}\label{subs-fd}
We now detail our finite difference variant of Algorithm~\ref{AlgolazyOFFO}. In the description, we  give an analytical expression of the finite difference stepsize $h$.

\begin{algorithm}
	\caption{Adaptive $p$th order Regularization with Finite Difference \texttt{ARpapprox-FD}}\label{ARpFD}	
	\begin{algorithmic}
		\Require Same as Algorithm~\ref{AlgolazyOFFO} with $\kappa_C = 1$. And set $\kappa_A = \frac{L_p}{2}$ and $\kappa_B = \frac{L_p}{2}$.
		\State
		Follow the steps Algorithm~\ref{AlgolazyOFFO} where
        \texttt{Tensor\_p} is set as 
		
		\If{$k \mod m =0$}
		define 
		\begin{equation}\label{hdef}
			h_k = \frac{\min\left( \sum_{i=1}^{m} \|s_{k-i}\|, \, 1\right)}{\sqrt{n}}
		\end{equation}
		and define the approximate tensor as
		\begin{equation}\label{approx$p$thtensor}
			A_k = \bigsum_{i=1}^{n} \left(\frac{\nabla_x^{p-1} f(x+h_ke_i) -  \nabla_x^{p-1} f(x)}{h_k} \right) \otimes e_i \quad \quad \textrm{ and } T_k = P_{{sym}}(A_k)
		\end{equation}
		\Else
		\begin{equation}\label{keepconstantDF}
			T_{k} = T_{k-1}
		\end{equation}
		\EndIf
	\end{algorithmic}
\end{algorithm} 

Observe that Algorithm~\ref{ARpFD} is a lazy method as it computes computing an approximation every $m$ steps. From our choice \eqref{hdef}, we define an implicit selection of $\kappa_A$ and $\kappa_B$, we specified their values to be consistent with \texttt{ARp-approx}.  First, we define the approximate $p$th-order tensor $A_k$ by using  finite differences. We apply the $P_{sym}$ operator (see page~3) to ensure that the estimate $T_k$ lies in $\mathbb{R}^{\otimes p_n}_{{sym}}$. Note that for $p=2$, $P_{{sym}}(A) = \frac{A^\intercal + A}{2}$. 
Moreover, { \texttt{ARpLazy} does not use any problem constant like $p$th-order Lipschitz constant $L_p$.}
{Note that, as the iterates approach a solution and the step size becomes increasingly small, $h_k$
	tends to zero, making the finite-difference approximation susceptible to round-off errors. A standard remedy is to impose a lower bound on the finite-difference step size, typically chosen as a function of the machine precision. Such a safeguard also imposes a lower limit on the attainable criticality tolerance, as analyzed in \cite{JinXie24} for stochastic gradient descent. However, once a minimum finite-difference step size is enforced, the bound \eqref{condTkapprox} no longer holds. Consequently, the convergence analysis must be extended to account for this modification, leading to a more involved theory in which the target accuracy $\epsilon$ is prescribed in advance. }
%and it is the standard method to form approximate Hessian in finite differences approximation, see \cite{doikov2023zerothorder,Grapiglia2021}.

Again, we now prove that the conditions stated in Section~\ref{presentation} apply to Algorithm~\ref{ARpFD}. 
\begin{lem}{FDoffo}
		Algorithm~\ref{ARpFD} is a valid instance of Algorithm~\ref{AlgolazyOFFO} with { $\kappa_A = \kappa_B =  \frac{L_p}{2}$ and $\kappa_C = 1$ in Condition~\ref{Conditionpth}.}
\end{lem}

\begin{proof}
	From \cite[Theorem~(A.7.3), (A.15)]{Cartis2022-wb} and taking $s = h_k e_i$ with $i \in \iibe{1}{n}$, we have that
	\[
	\|\nabla_x^{p-1} f(x+h_ke_i) - \nabla_x^{p-1} f(x) - h_k\nabla_x^{p} f(x)[e_i] \| \leq \frac{L_p h_k^2}{2}.
	\]
	Dividing by $h_k$ and rearranging, we derive that
	\[
		\|\frac{\nabla_x^{p-1} f(x+h_ke_i) - \nabla_x^{p-1} f(x)}{h_k} - \nabla_x^{p} f(x)[e_i] \| \leq \frac{L_p h_k}{2}.
	\]
	Now using the definition of the tensor norm and the definition of $P_{{sym}}$,  Cauchy-Schwarz inequality, the last inequality and the definition of $A_k$ \eqref{approx$p$thtensor}, we derive that
	\begin{align}\label{FDOresult}
	\|T_k - \nabla_x^p f(x_k) \| &\leq \|A_k - \nabla_x^p f(x_k)\| = \max_{\|s_i\| = 1, 1\leq i \leq p} |(A_k - \nabla_x^p f(x_k))[s_1, \dots s_n]| \nonumber  \\
	&= \max_{\|s_i\|=1, 1 \leq i \leq p-1, \sum_{i=1}^{n} \lambda_i^2 = 1} |\sum_{i= 1}^n \lambda_i \left(A_k - \nabla_x^p f(x_k)\right)[s_1, \dots s_{p-1}, e_i] | \nonumber \\ 
	&\leq \sqrt{n} \max_{i \in \iibe{1}{n} } \max_{\|s_i\|=1, 1 \leq i \leq p-1} | \left(A_k - \nabla_x^p f(x_k)\right)[s_1, \dots s_{p-1}, e_i]| \nonumber \\
	&= \sqrt{n} \max_{i \in \iibe{1}{n}}  	\|\frac{\nabla_x^{p-1} f(x+h_ke_i) - \nabla_x^{p-1} f(x)}{h_k} - \nabla_x^{p} f(x)[e_i] \| \nonumber \\
	 &\leq \frac{L_p\min\left( \sum_{i=1}^{m} \|s_{k-i}\|, \, 1\right)}{2}.
	\end{align}
	Now using \eqref{hdef}, we derive that \eqref{Tkmodmcomp} holds with $\kappa_A = \kappa_B = \frac{L_p}{2}$.
\end{proof}
We are now in position to state our complexity bound on the number of oracle calls of \texttt{ARp-FD}. 

\begin{theo}{ComplexARpFD}
Suppose that Assumptions 1--5 hold. Then  Algorithm~\ref{ARpFD} requires at most
\\
$\mathcal{O}\left(\max[\epsilon_1^{-(p+1)/p}, \, \, \epsilon_2^{-(p+1)/(p-1)}]\right)$ iterations and $\mathcal{O}\left(n\max[\epsilon_1^{-(p+1)/p}, \, \, \epsilon_2^{-(p+1)/(p-1)}]\right)$ evaluations of $\{ \nabla_x^i f\}_{i=1}^{p-1}$ to produce a vector $x_\epsilon \in \mathbb{R}^n$ such that $\|\nabla_x^1 f(x_\epsilon)\| \leq \epsilon$ and $\lambda_{\min}(\nabla_x^2 f(x_\epsilon)) \geq -\epsilon_2$. 
\end{theo}
\begin{proof}
	The first part is a consequence of both Lemma~\ref{FDoffo} and Theorem~\ref{complexbound}. And the second part comes from \eqref{approx$p$thtensor} as we call the $p-1$th tensor derivative $n$ times to compute $T_k$. 
\end{proof}

Therefore, Algorithm~\ref{ARpFD} provides an inexact and implementable $p$th-order free {lazy} method that uses finite differences to approximate the last order derivative. To the best of the authors' knowledge, no such result has been developed in the literature of tensor methods for $p \geq 3$, and it may open the avenue to practical high-order tensors. 
For $p=2$, \cite{doikov2023zerothorder} proposed a similar variant to our \texttt{ARp-FD} method and where, for a specific choice of $m$ in Algorithm~\ref{ARpFD}, they prove a bound $\mathcal{O}\left(\sqrt{n} \epsilon_1^{-3/2} + n \right)$ on the number of calls to the gradient. This suggests that the bound in Theorem~\ref{ComplexARpFD} can be improved by a more thorough analysis, in the light of the one done in \cite{doikov2023zerothorder}. This is beyond the scope of the present presentation and is left for future work. We now discuss the inclusion of  high-order secant updates in the calculation of the $p$th-order approximation.

%Note also that thanks to condition \eqref{Tkmodmcomp}, we can circumvent the need of exact $p$th order information and can use finite-deference of the inferior order to compute the inexact approximation required every $m$ step  \eqref{Tkmodmcomp}. Our work therefore extend  the result developed in \cite{doikov2023zerothorder} for tensor methods with $p \geq 3$. Indeed, we have shown that the standard $\mathcal{O}\left(\epsilon^{-(p+1)/p}\right)$ can be achieved by only using up to order $(p-1)$ derivative information. Additional analysis on the optimal choice of $m$ as done in \cite{doikov2023zerothorder} and the dependency on the problem's dimension $n$ usually present in complexity bound of DFO algorithms will not be discussed here and are out of the scope of the current presentation.  

\subsection{High-Order Secant  Updates} \label{subs-HOSU}
In this subsection, we show that  \eqref{updatecondition} allows the inclusion of the secant equation when updating $T_k$ for $k \mod m \neq 0$. We recall some notation introduced in \cite{Welzel2024}. As it is standard in quasi-Newton methods, we use the past information to update the $p$th-order tensor as 
\begin{equation}\label{HOSU}
	T_{k} = \argmin_{T \in \mathbb{R}^{\otimes p_n}_{sym}} \|(T-T_{k-1})[W_k]^p\|_F \text{ s.t. } T[s_{k-1}] = \widetilde{T}_{k-1}[s_{k-1}],
\end{equation}
where $W_k$ is a symmetric positive definite matrix. {The fact that the argmin is a singleton in the last inequality stems from \cite[Theorem~4.1]{Welzel2024}.}
The update rules covered by \eqref{HOSU} for $p=2$ {coincide with} the standard quasi-Newton formulas. When $W_k=I_n$, \eqref{HOSU} retrieves the PSB formula \cite{POWELL1970}, a symmetric version of the initial update rule proposed by Broyden \cite{Broyden1965}. For $W_k = \widetilde{B}_k^{-1/2} = (\int_{0}^1 \nabla_x^2 f(x_k + t s_k) \, dt)^{-1/2} $ (provided that the latter is well defined), we obtain the DFP update rule \cite{Fletcher1963,Davidon59}.

In the two next subsections, we  prove that the newly introduced high-order secant update (PSB and DFP) \eqref{HOSU} can be used to enhance the approximate $p$th order tensor between every $m$ steps. Before giving a detailed discussion on both PSB and DFP variants, let us prove that the update \eqref{HOSU} improves the estimate of $\widetilde{T}_{k-1}$ with respect to the weighted norm $\|.\|_{W_k}$.

\begin{lem}{HOSUTheorem}
	Let $k \geq 0$ and ($k \mod m \neq 0$). Suppose that each iteration $T_k$ is updated by \eqref{HOSU} with $W_k$ symmetric positive definite. Then
	\begin{equation}\label{improvWknorm}
			\|T_{k} - \widetilde{T}_{k-1}\|_{W_k} \leq \|T_{k-1} - \widetilde{T}_{k-1}\|_{W_k}. 
	\end{equation} 
\end{lem} 

\begin{proof}
	Recall that, from \cite[Theorem~4.1(d)]{Welzel2024}, for any $W_k$, the solution of \eqref{HOSU} can be written as
	\[
	(T_k - \widetilde{T}_{k-1})[W_k]^p = (T_{k-1} - \widetilde{T}_{k-1})[W_k]^p \left[I_n - \frac{v_k v_{k}^\intercal}{v_k^\intercal v_k}\right]^p,
	\]
	with $v_{k} = W_k^{-1} s_{k-1}$.
	Taking the operator norm $\|.\|$ in the last equality, \eqref{W2weighted}, and using that $\|I_n - \frac{v_k v_{k}^\intercal}{v_k^\intercal v_k} \| \leq 1$ gives the desired result. 
\end{proof}

Equipped with this last Lemma, we are now ready to illustrate our two algorithmic variants that employ \eqref{HOSU}, starting with the PSB variant.

\subsubsection{The PSB Update}

We now illustrate  two instantiations of the PSB variants for the Algorithm~\ref{AlgolazyOFFO}. The first variant  assumes a restart by using the exact $p$th order tensor which we will call \texttt{ARp-PSB Lazyrestart} and another variant where the restart is done is performed by finite differences called \texttt{ARp-PSB FDrestart}. The restart mechanism describes how the computation is done in order to satsify \eqref{Tkmodmcomp}. We will now describe the two subvariants in detail.

\begin{algorithm}
	\caption{Adaptive $p$th order regularization with PSB and Lazy update (\texttt{ARp-PSB Lazyrestart})}\label{ARpPSBLazy}	
	\begin{algorithmic}
		\Require Same as Algorithm~\ref{AlgolazyOFFO} with $\kappa_C = 1$. And define $\kappa_A =0$ and $\kappa_B =0$.
		\State
		Follow the steps of  Algorithm~\ref{AlgolazyOFFO} where \texttt{Tensor\_p} is set as:
		
		\If{$k \mod m =0$}
		\begin{equation}\label{PSBLazy}
			T_k = \nabla_x^p f(x_k),
		\end{equation}
		\Else
		\begin{equation}\label{HOSUupdatePSBLazy}
			T_k = \argmin_{T \in \mathbb{R}^{\otimes p_n}_{sym}} \|(T-T_{k-1})\|_F \text{ s.t. } T[s_{k-1}] = \nabla_x^{p-1}f(x_k) - \nabla_x^{p-1}f(x_{k-1}).
		\end{equation}
		\EndIf
	\end{algorithmic}
\end{algorithm}

\begin{algorithm}
	\caption{Adaptive $p$th order regularization with PSB and Finite Differences (\texttt{ARp-PSB FDrestart})}\label{ARpPSB}	
	\begin{algorithmic}
		\Require Same as Algorithm~\ref{AlgolazyOFFO} with $\kappa_C = 1$. And define  $\kappa_A = \frac{L_p}{2}$ and $\kappa_B = \frac{L_p}{2}$.
		\State
		Same as Algorithm~\ref{AlgolazyOFFO} but choose \texttt{Tensor\_p} as 
		
		\If{$k \mod m =0$}
		define
		\begin{equation}\label{hdefPSB}
			h_k = \frac{\min\left( \sum_{i=1}^{m} \|s_{k-i}\|, \, 1\right)}{\sqrt{n}},
		\end{equation}
		and define the approximate tensor as
		\begin{equation}\label{approx$p$thtensorPSB}
			A_k = \bigsum_{i=1}^{n} \left(\frac{\nabla_x^{p-1} f(x+h_ke_i) -  \nabla_x^{p-1} f(x)}{h_k} \right) \otimes e_i \quad \quad \textrm{ and } T_k = P_{{sym}}(A_k),
		\end{equation}
		\Else
		\begin{equation}\label{HOSUupdatePSB}
			T_k = \argmin_{T \in \mathbb{R}^{\otimes p_n}_{sym}} \|(T-T_{k-1})\|_F \text{ s.t. } T_{k}[s_{k-1}] = \nabla_x^{p-1}f(x_k) - \nabla_x^{p-1}f(x_{k-1}).
		\end{equation}
		\EndIf
	\end{algorithmic}
\end{algorithm}

\noindent
Let us now prove that Algorithm~\ref{ARpPSBLazy} and Algorithm~\ref{ARpPSB} are covered by our analysis.
\begin{lem}{PSBoffo}
	Both Algorithm~\ref{ARpPSBLazy} and Algorithm~\ref{ARpPSB} are valid variants of Algorithm~\ref{AlgolazyOFFO} with {$\kappa_C = 1$, $\kappa_A = \kappa_B =0$ in the first case (Algorithm~\ref{ARpPSBLazy}) and $\kappa_C = 1$,  $\kappa_A = \kappa_B = \frac{L_p}{2}$ in the second case (Algorithm~\ref{ARpPSB}).}
\end{lem}
\begin{proof}
For the analysis of the two restart mechanisms \eqref{PSBLazy} and \eqref{approx$p$thtensorPSB} the arguments of Lemma~\ref{proofArplazy} and Lemma~\ref{FDoffo} apply.  Now for $k \mod m \neq 0$, using \eqref{improvWknorm} with $W_k = I_n$ gives \eqref{updatecondition} with $\kappa_C = 1$ for both \eqref{HOSUupdatePSBLazy} and \eqref{HOSUupdatePSB}. 
\end{proof}
Note that for these two variants Algorithm~\ref{ARpPSB} and Algorithm~\ref{ARpPSBLazy}, the result of the two Theorem~\ref{lazyoffo} and Theorem~\ref{ComplexARpFD} apply. In the following, we  provide a discussion for both cases $p=2$ and $p=3$.   

\begin{corr}{PSBcomplex}
Suppose that Assumptions 1--3 hold and that $p=2$. Then Algorithm~\ref{ARpPSBLazy} requires at most $\mathcal{O}\left(\max[\epsilon_1^{-3/2}, \, \, \epsilon_2^{-3}]\right)$ evaluations of the gradient and the Hessian to produce a vector $x_\epsilon \in \mathbb{R}^n$ such that $\|\nabla_x^1 f(x_\epsilon)\| \leq \epsilon_1$ and $\lambda_{\min}(\nabla_x^2 f(x_\epsilon)) \geq -\epsilon_2$. The second PSB variant Algorithm~\ref{ARpPSB}  requires at most $\mathcal{O}\left(n\max[\epsilon_1^{-3/2}, \, \, \epsilon_2^{-3}]\right)$ evaluations of the gradient to achieve the same conditions.
\end{corr}
\begin{proof}
In both cases, it is obtained by using a similar reasoning as in Theorem~\ref{lazyoffo} or Theorem~\ref{ComplexARpFD}.
\end{proof}

From Corollary~\ref{PSBcomplex}, the \texttt{AR2-PSB FDrestart}, which integrates quasi-Newton updating and finite differences to estimate the Hessian, reaches a second-order stationary point using $\mathcal{O}\left(n\max[\epsilon_1^{-3/2}, \, \, \epsilon_2^{-3}]\right)$ evaluations of the gradient. This is the optimal bound for second-order methods up to term $n$. It is noteworthy that our analysis made no assumptions about the boundedness of the quasi-Newton approximation $B_k$. In fact, this was the main objective of the condition \eqref{addedcond} on the trial step $s_k$ and the refinement of the work \cite{GraJerToin24} in Section~\ref{analysis}. Alternatively, if we had removed the above condition and applied the results of the cited study directly, we would have been forced to assume that the negative curvature of $B_k$ remains bounded by a constant for all iterations. This is a rather restrictive assumption for quasi-Newton methods, since we know that $\|B_k\|$ can be { at worst bounded linearly with respect to $k$ (it is unknown whether the bound is tight or not), see \cite{Conn2000}.} Of course, the \texttt{ARp-PSB Lazyrestart} variant, where every $m$ steps performs an exact evaluation of the Hessian, has a better dependence on $n$ at the cost of requiring exact access to the Hessian. 
We now detail our result for the case $p=3$.

\begin{corr}{PSBpthree}
	Suppose that Assumptions 1--5 hold and that $p=3$. Then Algorithm~\ref{ARpPSBLazy} requires at most $\mathcal{O}\left(\max[\epsilon_1^{-4/3}, \, \, \epsilon_2^{-2}]\right)$ evaluations of the gradient, the Hessian and  third-order tensor to produce a vector $x_\epsilon \in \mathbb{R}^n$ such that $\|\nabla_x^1 f(x_\epsilon)\| \leq \epsilon$ and $\lambda_{\min}(\nabla_x^2 f(x_\epsilon)) \geq -\epsilon_2$. The second PSB variant Algorithm~\ref{ARpPSB}  requires at most $\mathcal{O}\left(n\max[\epsilon_1^{-4/3}, \, \, \epsilon_2^{-2}]\right)$ evaluations of the gradient and the Hessian to achieve the same termination conditions.
\end{corr}  

Note that the idea of using low-order information to approximate third-order tensor is not new and has already been considered in \cite{Schnabel1984,Schnabel1991}. In the former work, an interpolating model was constructed using past objective function information and no complexity analysis was performed in both papers. In our case, we use a combination of finite differences and high-order secant equation to compute the approximate third-order derivative. Note that these new high-order approximation can be combined with the new practical variants of third-order regularization methods \cite{WenqiCartis2025}.

\subsubsection{The DFP Update}

We now discuss in the next paragraph how we can incorporate the DFP update rule in Algorithm~\ref{AlgolazyOFFO}. We remark that the DFP update  would be retrieved in \eqref{HOSU} for all $W_k$ such that 
\begin{equation}\label{DFPcond}
W_k^{-T} W_k^{-1} s_{k-1} = y_{k-1} \eqdef \nabla_x^1 f(x_k) - \nabla_x^1 f(x_{k-1}),
\end{equation}
as shown in  \cite{Welzel2024}. Note also that the result would be the same if we have considered $ \alpha s_{k-1}$ instead of $s_{k-1}$ for any $\alpha \in \mathbb{R}$ in both \eqref{HOSU} with 
\[
W_k^{-T} W_k^{-1} \alpha s_{k-1} = y_{k-1} = \nabla_x^1 f(x_k) - \nabla_x^1 f(x_{k-1}),
\]
\noindent
 as proven in \cite[Theorem~4.1 (b)]{Welzel2024}. Note that since progress in the DFP approximation is measured in the weighted norm as shown in \eqref{improvWknorm}, we have to make sure that the constructed $W_k$ satisfying \eqref{DFPcond} has a bounded condition number so that the analysis of Algorithm~\ref{AlgolazyOFFO} applies. Fortunately, we can construct $W_k$ such that \eqref{DFPcond} applies with $W_k$ having a bounded condition number under suitable assumptions on $s_{k-1}$ and $y_{k-1}$. For the sake of convenience, we will drop the subscript $k$ for the next lemma since we will be focusing on a specific iteration. 
 
 \begin{lem}{constructWk}
 	Let $ L, \,  \mu, \, \varsigma > 0$ and suppose that $\mu \|s\|^2 \leq |s^\intercal y| \leq L \|s\|^2$ and $\|y\| \leq L \|s\|$.  Then, we can construct a matrix $W$ such that $W^{-T} W^{-1}s = y$ and 
 	\begin{equation}\label{kappaW}
 		\kappa(W) \leq \kappa_{\max} \eqdef \max\left[\sqrt{\frac{1}{\mu}}, \sqrt{ L + \frac{\varsigma + L^2}{\mu}} \sqrt{\max\left[\frac{1}{\varsigma}, \, 1, \, \frac{\varsigma + L^2}{\mu\varsigma}\right]} \right]
 	\end{equation}
 \end{lem}
 \begin{proof}
 The proof is deferred to  Appendix~\ref{Proofconstructwk}.
 \end{proof}
 
 We now discuss the conditions in Lemma~\ref{constructWk} imposed on $s_{k-1}$ and $y_{k-1}$ when constructing $W_k$. First note that $\mu \|s_{k-1}\|^2 \leq |s_{k-1}^\intercal y_{k-1}|$ is related to a local strong convexity of the function, a condition that can ensured in practice by using a linesearch with the Wolfe conditions after computing the quasi-Newton step, see \cite{NoceWrig06}. The other condition $\|y_{k-1}\| \leq L \|s_{k-1}\|$ requires only Lipschitz gradient continuity along the path of the specific iterate
 
 We now describe how the Algorithm~\ref{AlgolazyOFFO} can be combined with the DFP update rule. For the sake of convenience, we will only provide the variant that employs the finite differences restart.  
\begin{algorithm}
	\caption{Adaptive $p$th order regularization with DFP and Finite Differences \texttt{ARp-DFP FDrestart}}\label{ARpDFP}	
	\begin{algorithmic}
		\Require Same as Algorithm~\ref{AlgolazyOFFO} with additional hyper-parameter $L, \, \mu, \, \varsigma > 0$ $\kappa_C = \kappa_{\max}^p$ with $\kappa_{\max}$ defined in \eqref{kappaW}. Define implicitly $\kappa_A = \frac{L_p}{2}$ and $\kappa_B = \frac{L_p}{2}$.
		\State
		Follow the steps of Algorithm~\ref{AlgolazyOFFO} where \texttt{Tensor\_p} is set as 
		\If{$k \mod m =0$}
		define
		\begin{equation}\label{hdefDFP}
			h_k = \frac{\min\left( \sum_{i=1}^{m} \|s_{k-i}\|, \, 1\right)}{\sqrt{n}}
		\end{equation}
		and define the approximate tensor as
		\begin{equation}\label{approx$p$thtensorDFP}
			A_k = \bigsum_{i=1}^{n} \left(\frac{\nabla_x^{p-1} f(x+h_ke_i) -  \nabla_x^{p-1} f(x)}{h_k} \right) \otimes e_i \quad \quad \textrm{ and } T_k = P_{{sym}}(A_k)
		\end{equation}
		\Else
		\If{$\mu \|s_{k-1}\|^2 \leq |s_{k-1}^\intercal y_{k-1}|$ and $\|y_{k-1}\| \leq L \|s_{k-1}\|$}
		\begin{equation}\label{HOSUupdateDFP}
			T_k = \argmin_{T \in \mathbb{R}^{\otimes p_n}_{sym}} \|(T-T_{k-1})[W_k]^p\|_F \text{ s.t. } T_{k}[s_{k-1}] = \nabla_x^{p-1}f(x_k) - \nabla_x^{p-1}f(x_{k-1})
		\end{equation}
		with $W_k$ computed as in Lemma~\ref{constructWk}.
		\Else
		\begin{equation}\label{keepconstantDFP}
			T_k = T_{k-1}
		\end{equation}
		\EndIf
		\EndIf
	\end{algorithmic}
\end{algorithm}

We now analyze the \texttt{ARp-DFP FDrestart} and see that it verifies the hypothesis of Condition~\ref{Conditionpth}.
\begin{lem}{ARpDFPproof}
	Algorithm~\ref{ARpDFP} is a valid instance of Algorithm~\ref{AlgolazyOFFO} with the additional hyper-parameters $L, \, \mu, \, \varsigma$, 
{with $\kappa_C = \kappa_{\max}^p$ where $\kappa_{\max}$ is defined in \eqref{kappaW} and $\kappa_A = \kappa_B =  \frac{L_p}{2}$.}
\end{lem}
\begin{proof}
The reasoning for both \texttt{ARp-FD} and \texttt{ARp-PSB} upholds when $k \mod m = 0$. Now for $k \mod m \neq 0$ and when using \eqref{HOSUupdateDFP} to update $T_k$, \eqref{improvWknorm} yields that
\[
	\|T_{k} - \widetilde{T}_{k-1}\|_{W_k} \leq \|T_{k-1} - \widetilde{T}_{k-1}\|_{W_k} 
\]  
Now using \eqref{CompairsonWnorm} and that $\kappa(W_k) \leq \kappa_{\max}$, we obtain 
\[
\|T_{k} -\widetilde{T}_{k-1} \| \leq \kappa_C \|T_{k-1} - \widetilde{T}_{k-1}\|,
\]
with $\kappa_C = \kappa_{\max}^p$. Remark that the latter inequality is still valid when \eqref{keepconstantDFP} is used since $\kappa_C \geq 1$. Hence \eqref{updatecondition} holds.
\end{proof}
{Remark \texttt{ARp-DFP FD-restart} does not utilize any problem dependent constant (such as $L_p$) and that $L, \, \mu, \, \varsigma$ are hyper-parameters given by the user and so $\kappa_{\max}$ defined in \eqref{kappaW} is known. Observe also that by the choice of computing $T_k$ as in \texttt{ARp-DFP FD-restart}, the  values of $\kappa_A$ and $\kappa_B$ are set to  $\frac{L_p}{2}$. Again, $L_p$ or any related problem's constant does not appear as input of Algorithm~\ref{ARpDFP}.}
For the newly introduced \texttt{ARp-DFP FD-restart} and its analogous  lazy variant (not shown), the results of the previous subsection (Corollary~\ref{PSBcomplex} and Corollary~\ref{PSBpthree}) hold with a similar conclusion. Although the fact that $\kappa_C = \kappa_{\max}^p$ is a large constant may theoretically lead to a worse dependence in $\mathcal{O}$, we believe that numerically \texttt{ARp-DFP} may give better results than \texttt{ARp-PSB} variants for $p \geq 3$, since \eqref{HOSUupdateDFP} is scale-invariant. 

\section{Numerical Illustration}\label{numeric-s}

{We now illustrate the behavior of our proposed algorithm for $p=3$ on the 35 problems curated by Mor\'e, Garbow and Hillstrom \cite{MGH81}, hereafter referred to as MGH. Our implementation extends the one developed in \cite{Cartisetal24}, where an efficient and robust implementation of third-order tensor methods is presented. Our numerical study is divided into two parts. Firstly, we consider deterministic problems, to illustrate that refreshing the derivative approximation every $m$ iterations, while using PSB updates between successive recomputations, can reduce the computational cost associated with evaluating third-order tensor derivative. Secondly, we present experiments on noisy problems to highlight the benefits of the OFFO approach, which is specifically designed for this class of problems, as demonstrated in \cite{GraJerToin24,OFFO-ARp}. }

{
	Since the tested methods may employ different orders of derivatives at each iteration, we assign a unit cost to a function evaluation, a cost of $n$ to a gradient evaluation, a cost of $n^2$
	to a Hessian evaluation, and a cost of $n^3$
	to a third-order tensor evaluation. All other operations, including linear algebra computations and subproblem solves, are assumed to have negligible cost for the purpose of this metric. This weighting provides a meaningful measure of the overall evaluation cost. In addition, we report the total number of iterations, which is equivalent to the number of subproblem solves, as a complementary performance metric.}

\subsection{Lazy and HOSU in the Deterministic Case}\label{subseclazyHOSUexact}

{In this subsection, we evaluate the performance of the proposed algorithms using the standard update of $\sigma_k$ based on function values, as in \cite{Cartisetal24}. As a baseline, we employ the highly efficient adaptive third-order method \texttt{AR3-Interp\textsuperscript{+}} introduced in \cite{Cartisetal24}, which we denote here as \texttt{AR3-Full} for the sake of simplicity. Note that this variant of \texttt{AR3} incorporates a prerejection mechanism of unreasonable steps.  The aim of the experiments here is to assess the benefits of recomputing the third-order tensor only periodically while relying on secant updates between successive recomputations. To this end, we compare five methods. The first is the well-tuned \texttt{AR3-Full} implementation from \cite{Cartisetal24}. Building on this implementation, we develop four new variants:
	\[
	\texttt{AR3-$\text{LAZY}_n$}, \, \,
	\texttt{AR3-$\text{LAZY}_\infty$}, \, \,
	\texttt{AR3-$\text{PSB}_n$}, \, \,
	\texttt{AR3-$\text{PSB}_\infty$}.
	\]
	In these variants, the subscript specifies the frequency at which the tensor approximation is refreshed: $n$ denotes a recomputation every $n$ iterations (where $n$ is the problem dimension), while $\infty$ indicates that the approximation is never recomputed. Between successive refreshes, the methods with the \texttt{PSB} suffix update the tensor approximation using \eqref{HOSUupdatePSBLazy} after successful iterations, whereas the \texttt{LAZY} variants simply keep the last computed tensor unchanged.
	For the $\infty$ case, \texttt{AR3-$\text{LAZY}_\infty$} uses the initial tensor $\nabla_x^3 f(x_0)$ throughout the optimization, while \texttt{AR3-$\text{PSB}_\infty$} starts from the zero third-order tensor and then updates the approximation using \eqref{HOSUupdatePSBLazy}. 
	In all experiments of this subsection, we set the termination tolerance to $\epsilon = 10^{-6}$. The results are presented using the performance profiles of Dolan and Moré \cite{DoleMore02} for both performance metrics.}

\begin{figure}
	\centering
	\includegraphics[width = 0.8\linewidth]{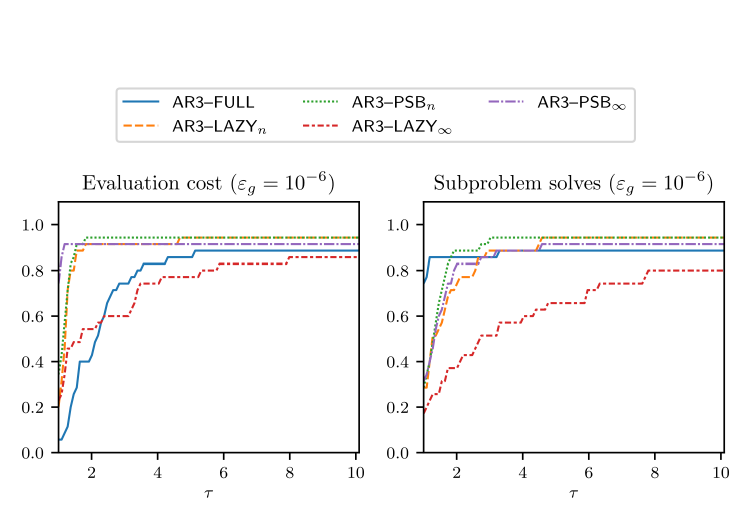}
	\caption{Performance profile for various \texttt{PSB} and \texttt{Lazy} variants for different memory size on the MGH problem collection. 
		%the x-axis denotes the ratio of the cost  considered the budget given in terms of the considered metric count, and the y-axis gives the percentage of solved instances.
	}
	\label{LazyPSBexactcase}
\end{figure}  

{In terms of the number of iterations, \texttt{AR3-FULL} is the most efficient method, requiring the fewest iterations to converge. This advantage, however, comes at the expense of a significantly higher evaluation cost, as illustrated in the right panel of Figure~\ref{LazyPSBexactcase}. In contrast, \texttt{AR3-$\text{LAZY}_\infty$} performs poorly according to both metrics, as relying solely on the initial third-order tensor leads to increasingly inaccurate local models throughout the optimization process. Surprisingly, \texttt{AR3-$\text{PSB}_\infty$} exhibits good robustness while maintaining a low evaluation cost, since it never computes an exact third-order tensor. Finally, \texttt{AR3-$\text{LAZY}_n$} and \texttt{AR3-$\text{PSB}_n$} achieve comparable performance, with a slight advantage for \texttt{AR3-$\text{PSB}_n$}. This behavior is expected, as \texttt{AR3-$\text{PSB}_n$} improves its tensor approximation through PSB updates between successive recomputations. }

% All methods with PSB in their suffix employ \eqref{HOSUupdatePSBLazy} between every exact computation of the exact tensor derivative at varied frequency. The suffix after the hypen shows the frequency of the recomputation of the exact tensor. $1$ is for an exact evaluation after every succesful step, $n$ is for 
% \texttt{LazyAR3$-\infty$} evalutes the third-order tensor derivative at the first iteration and then keeps it fixed whereas 	\texttt{LazyAR3$-n$} updates the used Tenosr by computing the third derivative every $n$ steps.           

\subsection{Noisy Lazy and HOSU in the OFFO case}\label{subseclazyHOSUnoisy}
{
	\subsubsection{Practical $\sigma_k$ update rules}
	For the noisy case, directly employing the update rule \eqref{sigmakupdate} to choose $\sigma_k$ may lead  suboptimal performance, since $\sigma_k$ would only keep increasing monotonically. Instead, we follow the implementation details given in \cite{OFFO-ARp}. In detail, at each iteration, $\sigma_k$ was  chosen as 
	\begin{equation}\label{sigmakpractise}
		\sigma_{k} \eqdef \max(\vartheta \nu_k, \xi_k \mu_{1,k}).
	\end{equation}
	$\nu_k$ now tracks the dynamics of \eqref{sigmakupdate} and is chosen as done in \cite{OFFO-ARp}, $\vartheta$ is a small positive constant, $\xi_k$ is a variable in $[\vartheta , 1]$  (see \cite[Section~5]{OFFO-ARp}) and $\mu_{1,k}$ is a constant derived from the past iterate such that $\mu_{1,k} \leq L_p$. For its expression, see \cite[Equation(2.8)]{OFFO-ARp}. The objective of this update rule is to better track local curvature changes and only employ the conservative \eqref{sigmakupdate} when needed. We set also $\sigma_{0}  = 24 \|g_0\|$.
}

\noindent
{
	Although the update rule \eqref{sigmakpractise} proved sufficient in the second-order setting, we observed that, for the third-order method, it may occasionally produce excessively large steps. As a result, the iterates can leave the region where the local model is accurate and enter a regime in which the optimization process becomes unstable. Motivated by \cite{BirginGardenghiMartSantos19}, we therefore incorporate a prerejection strategy that rejects trial steps whenever they predict an unreasonably large decrease or are excessively large in norm. Since the prerejection mechanism proposed in \cite{BirginGardenghiMartSantos19} is function-free, it can be seamlessly integrated into the OFFO framework. 
	Accordingly, $\sigma_{k}$ is now chosen as 
	\begin{equation}\label{sigmakpractisethird}
		\sigma_{k} \in [\vartheta_{\rm low} \nu_k, \xi_k \mu_{1,k}],
	\end{equation}
	where $\vartheta_{\rm low} = 1$.}
{We start by computing an initial step with $ \sigma_{k} = \max(\vartheta_{\rm k} \nu_k, \xi_k \mu_{1,k})$, if it is prejected, we double $\vartheta$  and we recompute a new trial step with a new $\sigma_k = \max(2*\vartheta_{k} \nu_k, \xi_k \mu_{1,k})$. Otherwise, (the step is accepted), the new $\vartheta_{k+1}$ is set as $\min(1, 0.5*\vartheta_k)$.
}

\noindent
{
	Note that this modified update rule \eqref{sigmakpractisethird} remains covered by the convergence analysis of \cite{OFFO-ARp} as $\vartheta_k \geq 1$ for all iterations. We denote this enhanced version of \texttt{OFFOAR$3$} by \texttt{AR3-OFFO\textsuperscript{+}}. We consider two variants of \texttt{AR3-OFFO\textsuperscript{+}}: \texttt{AR3-OFFO\textsuperscript{+}-FULL} and \texttt{AR3-OFFO\textsuperscript{+}-PSB\textsubscript{n}}, where the latter employs the same PSB updating strategy described in Subsection~\ref{subseclazyHOSUexact}. We compare these methods with their \texttt{AR3-Interp\textsuperscript{+}} counterparts which we denominate as \texttt{AR3}. Namely, we consider both \texttt{AR3-FULL} and \texttt{AR3-PSB\textsubscript{n}}. 
	For the noisy case, each individual entry of the function, gradient, Hessian and tensor is perturbed by adding $15\%$ relative noise \footnote{As an example for approximate gradient, $\left[ (\nabla_x^1 f(x_k))_i \times (1+ 0.15*\mathcal{U}\left(-1,1\right) )\right]_{i=1}^n$}.
	For these experiments, we set the stopping tolerance to $\epsilon = 10^{-3}$.
	\subsubsection{Numerical results for the noisy case}
	\begin{figure}
		\centering
		\includegraphics[width = 0.8\linewidth]{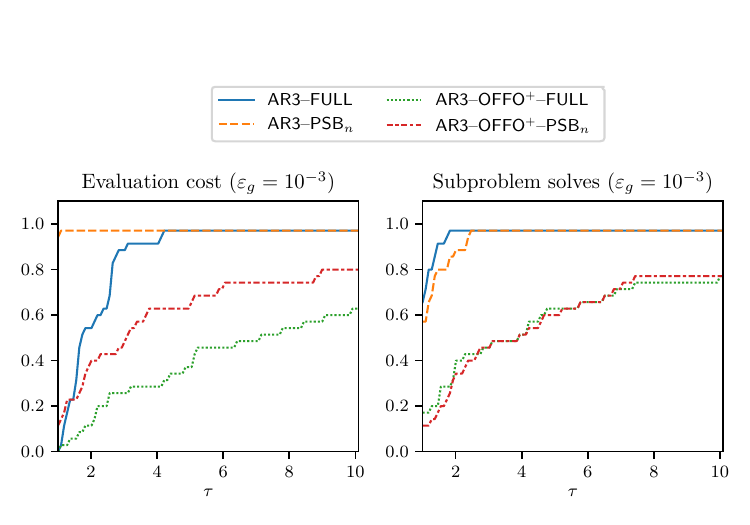}
		\caption{Performance profile with  two different metrics  for \texttt{FULL} and \texttt{PSB\textsubscript{n}} methods in the \texttt{OFFO\textsuperscript{+}} and standard \texttt{AR3} regimes in the exact case on the MGH problem collection. . 
			%the x-axis denotes the ratio of the cost  considered the budget given in terms of the considered metric count, and the y-axis gives the percentage of solved instances.
		}
		\label{OFFOexactcase}
	\end{figure} 
}

{
	As expected and as already illustrated in \cite{OFFO-ARp}, the performance of the \OFFOplus is dramatically worse than the standard \texttt{AR3} methods in the exact case as seen from Figure~\ref{OFFOexactcase}. Indeed, \OFFOplus methods even with some clever modifications, are still conservative and can't adapt to local curvature by decreasing aggressively the regularization parameter (at variance with standard \texttt{AR3}).
	We give now the results for the noisy case.
}
\begin{figure}
	\centering
	\includegraphics[width = 0.8\linewidth]{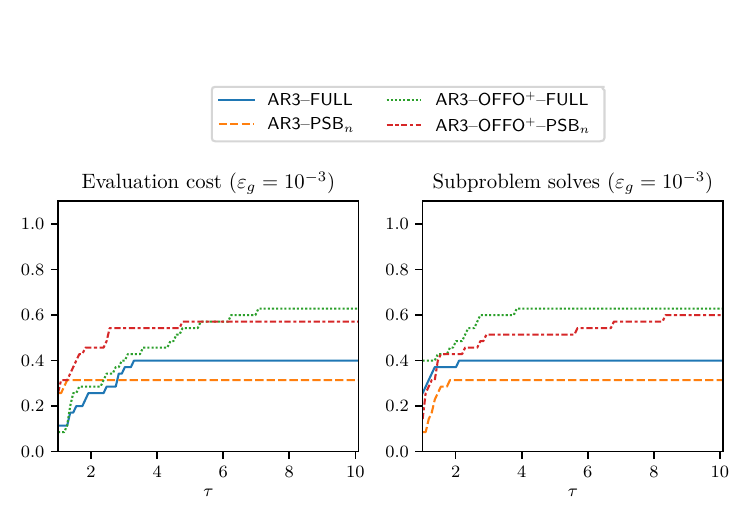}
	\caption{Performance profile with two  different metrics  for \texttt{FULL} and \texttt{PSB\textsubscript{n}} methods in the \texttt{OFFO\textsuperscript{+}} and standard \texttt{AR3} regimes with 15\% relative noise on the MGH problem collection. . 
		%the x-axis denotes the ratio of the cost  considered the budget given in terms of the considered metric count, and the y-axis gives the percentage of solved instances.
	}
	\label{OFFOnoisycase}
\end{figure} 

{As expected the performances of the \texttt{AR3} variants deteriorate significantly in the noisy setting compared with the deterministic case  (see Figure~\ref{OFFOexactcase} and Figure~\ref{OFFOnoisycase}). However, \texttt{AR3-OFFO\textsuperscript{+}} are still robust in the noisy case and there is only roughly a $10\%$ decrease in the number of solved instances. \texttt{AR3-OFFO\textsuperscript{+}-PSB\textsubscript{n}} exhibits slightly less reliability than the method that utilizes the noisy third-order derivative but comes with less  computational cost. It is worth emphasizing that the Hessian evaluations used in the PSB updates are themselves noisy. Nevertheless, the OFFO framework is sufficiently robust to compensate for these inaccuracies, resulting in competitive overall performance.
}

\section{Conclusion and Perspectives}\label{conclu-s}

We have proposed a  $p$th-order adaptive regularization framework where the approximation is restricted only to the $p$th-order tensor. The algorithm is adaptive and the inexactness is handled in a similar way to \cite{GraJerToin24} by avoiding the computation of the function values. The approximation is divided into two phases. An inexact evaluation of the $p$th-order tensor every $m$ step, which can be done either in a lazy way \cite{DoikChayJag23} or in a finite difference way with an explicit discretization step \cite{doikov2023zerothorder}. For other iterations, we can improve the approximate $p$th-order tensor by using the recently proposed high-order secant update  in \cite{Welzel2024}. {Initial experiments for $p=3$ highlight the merit of our method in both deterministic and noisy cases. In the former case, some of our developed variants exhibit close performance to the full third-order tensor method. In the latter, noisy case, a proposed variant of refreshed PSB yields good performance in relative noise settings.}

In the case $p=2$, we proposed a global complexity analysis of the two quasi-Newton variants PSB and DFP. For these two variants and using finite differences every $m$ steps, we proved that $(n \max(\epsilon_1^{-3/2}, \epsilon_2^{-3}))$ calls to the gradient oracle are sufficient to reach an $(\epsilon_1, \, \epsilon_2)$ second-order stationary point. Our theory also provides how to devise an optimal third-order tensor $(p=3)$ using only low-order information, thus paving the way for more practical third-order tensors \cite{CartisZhu25,Cartisetal24}.

Building on this work, various research directions can be pursued. Numerical tests of this new high-order approximate tensor method, be it in the case of $p=2$ or $p=3$, can be considered such as \cite{Cartisetal24} for all the different combinations of the $p$th-order approximation; be it with high-order secant update, finite differences, or exact computation every $m$ steps. A finer analysis of the optimal choice of $m$, as done in \cite{doikov2023zerothorder}, can be pursued to enhance the  dependence on $n$ for the finite difference variant. In addition to better understanding practical variants of quasi-Newton methods, replacing the restart mechanism every $m$ steps can be considered.

\section*{Acknowledgments}
The work of all three authors was supported by the Hong Kong Innovation and Technology Commission under the InnoHK initiative (Project CIMDA). Prof. Cartis additionally acknowledges the support of the UK EPSRC through grant EP/Y028872/1, titled "Mathematical Foundations of Intelligence: An 'Erlangen Programme' for AI."

%\be

%\newpage

{ \footnotesize
\bibliographystyle{plain}
\bibliography{referencesAll}
}

\appendix
\section{Missing Proofs of Section~\ref{analysis}}

\subsection{Proof of Lemma~\ref{boundbetak}}\label{firstlemmaproof}
\begin{proof}
	We consider two values of $\alpha \in \{\frac{p+1}{p-1}, \, \frac{2}{p-1} \}$. From the triangle inequality, the fact that $(x+y)^\sfrac{p+1}{p-1} \leq 2^\sfrac{2}{p-1} ( x^\sfrac{p+1}{p-1} +y^\sfrac{p+1}{p-1})$
	for $x, y \geq 0$ that $-\lambda_{\min}(A) \leq - \lambda_{\min}(A-B) - \lambda_{\min}(B)$ for any symmetric matrices $A,B$  and \eqref{addedcond}, we obtain that
	\begin{align*}
		\frac{\beta_{k+1}^\sfrac{p+1}{p-1}}{\sigma_{k+1}^\alpha} &= \frac{\max[0, - \lambda_{\min}(\nabla_x^2 f(x_{k+1}) )]^\sfrac{p+1}{p-1}}{\sigma_{k+1}^\alpha} \\
		&\leq  \frac{(\max[0, - \lambda_{\min} (\nabla_x^2 f(x_{k+1}) - \nabla_s^2\overline{ T_{f,p}}(x_k,s_k) )] + (\max[0, - \lambda_{\min}( \nabla_s^2\overline{ T_{f,p}}(x_k,s_k) )] )^\sfrac{p+1}{p-1}}{\sigma_{k+1}^\alpha}  \\
		&\leq 2^\sfrac{2}{p-1} \left(\frac{\|\nabla_x^2 f(x_{k+1}) - \nabla_s^2\overline{ T_{f,p}}(x_k,s_k)\|^\sfrac{p+1}{p-1}}{\sigma_{k+1}^\alpha}
		+ \frac{(\theta_2 \sigma_k)^\sfrac{p+1}{p-1} }{(p-1)!^\sfrac{p+1}{p-1} \sigma_{k+1}^\alpha} \|s_k\|^{p+1}
		\right).  
	\end{align*}
	Using now \eqref{Hessbounderror} in the previous inequality, we derive that
	\begin{equation}\label{usefulhess}
		\frac{\beta_{k+1}^\sfrac{p+1}{p-1}}{\sigma_{k+1}^\alpha} \leq 2^\sfrac{2}{p-1} \left( \kappa_1 \frac{\|s_k\|^{p+1}}{\sigma_{k+1}^\alpha} + \kappa_2 \frac{\xi_k}{\sigma_{k}^\alpha} + \frac{(\theta_2 \sigma_k)^\sfrac{p+1}{p-1} }{(p-1)!^\sfrac{p+1}{p-1} \sigma_{k+1}^\alpha} \|s_k\|^{p+1} 	\right).
	\end{equation}

	We first prove \eqref{crucialhess} and start with $\alpha = \frac{p+1}{p-1}$.
	Using that $\xi_k = \sum_{j= 1}^{2m-1} \|s_{k-j}\|^{p+1}$, that
	$\|s_j\|^{p+1} = \frac{\sigma_{j+1} - \sigma_{j}}{\sigma_{j}}$ for $j
	\in \iibe{k-2m+1}{k}$, \eqref{usefulhess}, and also that
	$\sigma_{k}$ is non-decreasing, also $\sigma_{k} \geq \sigma_0$ ,
	$\sigma_{k - 2m+1} \geq \frac{\sigma_0}{2^{2m-1}}$ for $k \geq 0$,  both facts
	resulting from \eqref{sigmakupdate} and \eqref{sigmajnegdef},
	\begin{align*}
		\frac{\beta_{k+1}^\sfrac{p+1}{p-1}}{\sigma_{k+1}^\sfrac{p+1}{p-1}} &\leq \twoopptwo \Biggl(\kappa_1 \frac{\sigma_{k+1}-\sigma_{k}}{\sigma_{k+1}} \frac{1}{\sigma_{k+1}^\sfrac{2}{p-1} \sigma_{k}} + \kappa_2 \sum_{j=k-2m+1}^{k-1}  \frac{\sigma_{j+1}-\sigma_{j}}{\sigma_{j+1}} \frac{1}{\sigma_{j} \sigma_k^\sfrac{2}{p-1}} + \frac{\theta_2 ^\sfrac{p+1}{p-1} }{(p-1)!^\sfrac{p+1}{p-1}} \frac{\sigma_{k}^\sfrac{2}{p-1}}{\sigma_{k+1}^\sfrac{2}{p-1}} \frac{\sigma_{k+1} - \sigma_{k}}{\sigma_{k+1}}   	\Biggr) \\
		&\leq \frac{\twoopptwo \kappa_1}{\sigma_{0}^\sfrac{p+1}{p-1}} \frac{\sigma_{k+1}-\sigma_{k}}{\sigma_{k+1}} + \kappa_4  \sum_{j=k-2m+1}^{k-1}  \frac{\sigma_{j+1}-\sigma_{j}}{\sigma_{j+1}}+ \frac{\twoopptwo \theta_2^\sfrac{p+1}{p-1}}{(p-1)!^\sfrac{p+1}{p-1}} \frac{\sigma_{k+1}-\sigma_{k}}{\sigma_{k+1}}
	\end{align*}
	where $\kappa_4$ is defined in \eqref{kappa34def}. Hence we obtain the first result of the lemma, \eqref{crucialhess} with $\kappa_3$ defined in \eqref{kappa34def}.
	
	\noindent
	Consider now the case $\alpha = \frac{2}{p-1}$. Again using the same arguments as in the proof of \eqref{crucialhess}, we deduce that,
	\begin{align*}
		\frac{\beta_{k+1}^\sfrac{p+1}{p-1}}{\sigma_{k+1}^\sfrac{2}{p-1}} &\leq \twoopptwo \Biggl(\kappa_1 (\sigma_{k+1}-\sigma_{k}) \frac{1}{\sigma_{k+1}^\sfrac{2}{p-1} \sigma_{k}} + \kappa_2 \sum_{j=k-2m+1}^{k-1}  (\sigma_{j+1}-\sigma_{j}) \frac{1}{\sigma_{j} \sigma_k^\sfrac{2}{p-1}} + \frac{\theta_2 ^\sfrac{p+1}{p-1} }{(p-1)!^\sfrac{p+1}{p-1}} \frac{\sigma_{k}^\sfrac{2}{p-1}}{\sigma_{k+1}^\sfrac{2}{p-1}} (\sigma_{k+1} - \sigma_{k}) 	\Biggr) \\
		&\leq \frac{\twoopptwo \kappa_1}{\sigma_{0}^\sfrac{p+1}{p-1}} (\sigma_{k+1}-\sigma_{k})+ \kappa_4  \sum_{j=k-2m+1}^{k-1}  \sigma_{j+1}-\sigma_{j}+ \frac{\twoopptwo \theta_2^\sfrac{p+1}{p-1}}{(p-1)!^\sfrac{p+1}{p-1}} (\sigma_{k+1}-\sigma_{k}).	
	\end{align*}
	Rearranging the last inequality gives the last result of the lemma. 
\end{proof}
\subsection{Proof of Lemma~\ref{skbound}}\label{firstappendix}
\begin{proof}

	Using the decrease of the model \eqref{modelconditions}, the definition of \eqref{xikdef} and the fact that $T_k=B_k$ when $(p=2)$, 
	\begin{align*}
		\frac{\sigma_{k}}{(p+1)!} \|s_k\|^{p+1} &\leq -g_k^\intercal s_k - \sum_{i=2}^{p-1} \frac{1}{i!} \nabla_x^i f(x_k)[s_k]^i - \frac{1}{p!} T_p[s_k]^p \\ 
		&= -g_k^\intercal s_k - \sum_{i=2}^{p-1} \frac{1}{i!} \nabla_x^i f(x_k)[s_k]^i - \frac{1}{p!} T_p[s_k]^p \mathds{1}_{p\geq3} - \frac{1}{2} s_k^\intercal B_k s_k \mathds{1}_{p=2} \\
		&\leq \|g_k\| \|s_k\| + \frac{1}{2} \xi_k \|s_k\|^2 - \sum_{i=3}^{p-1} \frac{1}{i!} \nabla_x^i f(x_k)[s_k]^i - \frac{1}{p!} T_p[s_k]^p \mathds{1}_{p\geq3}
	\end{align*}
	Using  the bound in Assumption~\ref{assum4} and the one in \eqref{Tkbound} (since Assumption~\ref{assum5} holds) and tensor norm properties, we derive that
	\[
	\frac{\sigma_{k}}{(p+1)!} \|s_k\|^{p+1}  \leq \|g_k\| \|s_k\| + \frac{1}{2} \xi_k \|s_k\|^2 + \sum_{i=3}^{p-1} \frac{\kappa_{high}}{i!} \|s_k\|^{i} + \frac{(2m+1)\kappa_p}{p!} \|s_k\|^{p} \mathds{1}_{p\geq3} 
	\]
	Applying now the Lagrange bound for polynomial roots \cite[Lecture~VI, Lemma~5]{Yap99} with $x=\|s_k\|^{p+1}$, $n=p+1$, $a_0=0$, $a_1 = \|g_k\|$, $a_2 = \frac{1}{2} \xi_k$, $a_i = \frac{\kappa_{high}}{i!}$ for $i \in \iibe{3}{p-1}$, $a_p = \frac{(2m+1)\kappa_p \indica{p\geq3}}{p!}$, $a_{p+1} =\frac{\sigma_{k}}{(p+1)!} $.  we know that the equation $\sum_{i=0}^{n} a_i x_i$ admits at least one positive root, and we may thus derive 
	
	\begin{align*}
		\|s_k\| \leq 2 \max \Biggl[ \, &\left(\frac{(p+1)!\|g_k\|}{\sigma_{k}}\right)^\sfrac{1}{p}, \, \, \left(\frac{(p+1)! \beta_k}{2 \sigma_{k}} \right)^\sfrac{1}{p-1}, \, \\ &\max_{i \in \iibe{3}{p-1} } \left(\frac{\kappa_{high} (p+1)!}{i! \sigma_{k}}\right)^\sfrac{1}{p-i+1} , \frac{(2m+1)(p+1) \kappa_p \indica{p\geq3}}{\sigma_{k}}  \Biggr]
	\end{align*}
	Using now that $\sigma_{k} \geq \sigma_{0}$ and the definition of \eqref{etadef} yields \eqref{skboundneg}. We now turn to proof \eqref{skboundizycase}. From
	\eqref{sigmakupdate}, and the fact that $\|s_k\|^{p+1} \indica{\|s_k\|
		\leq 2\eta} \leq (2\eta)^{p+1}$, we have that
	\[
	\sigma_{k+1} \indica{\|s_k\| \leq 2\eta} = \sigma_{k} \indica{\|s_k\| \leq 2\eta} + \|s_k\|^{p+1} \sigma_{k } \indica{\|s_k\| \leq 2\eta} \leq \sigma_{k} \indica{\|s_k\| \leq 2\eta} \left( 1 + (2\eta)^{p+1} \right),
	\]
	which yields that
	\[
	\frac{\sigma_{k+1} \indica{\|s_k\| \leq 2\eta}}{1 + (2\eta)^{p+1}} \leq \sigma_{k} \indica{\|s_k\| \leq 2\eta}.
	\]
	
	Multiplying both sides of the previous inequality by $\|s_k\|^{p+1}$,
	adding $\sigma_{k} \indica{\|s_k\| \leq 2\eta}$, and  using identity
	\eqref{sigmakupdate}, we derive that 
	\begin{align*}
		\sigma_{k} \indica{\|s_k\| \leq 2\eta} + \frac{\sigma_{k+1}}{1+(2\eta)^{p+1}} \|s_k\|^{p+1} \indica{\|s_k\| \leq 2\eta}   \leq \indica{\|s_k\| \leq 2\eta} \left( \sigma_{k} + \sigma_{k}  \|s_k\|^{p+1} \right) = 
		\sigma_{k+1} \indica{\|s_k\| \leq 2\eta}.
	\end{align*}
	Now rearranging the last inequality  yields \eqref{skboundizycase}.
\end{proof}

\subsection{Proof of Lemma~\ref{chiklemma}}\label{thirdproof}
\begin{proof}
	We will at first only focus on $p=2$ $(T_k = B_k)$.
	Using now the definition of \eqref{xikdef},  \eqref{chikdef}, \eqref{trueneg}, that $(x+y)^3 \leq 4 (x^3 + y^3)$, we derive that
	\begin{align}\label{firststep}
		\frac{\chi_k^3}{\sigma_{k}^3} &\leq \frac{( \max[0,-\lambda_{\min}(\nabla_x^2 f(x_k) - B_k)] + \max[0,-\lambda_{\min}(B_k)])^3}{\sigma_{k}^3} \nonumber \\ 
		&\leq 4 \frac{\|\nabla_x^2 f(x_k) - B_k\|^3}{\sigma_{k}^3} + 4  \frac{\beta_k^3}{\sigma_{k}^3},
	\end{align}
	where we have used that $-\lambda_{\min}(\nabla_x^2 f(x_k) - B_k) \leq \|\nabla_x^2 f(x_k) - B_k\|$. Now using \eqref{Tkmodmcomp} for $k=0$ and the fact that $\|s_{-1}\| = \cdots = \|s_{-m}\| = 1$ yields \eqref{xikzero} when $p=2$.  Note that \eqref{xikzero} is still valid for $p \geq 3$ since $\chi_0 = \max[0,-\lambda_{\min}(\nabla_x^2 f(x_0))]$.
	
	We provide now a bound on $\frac{\|\nabla_x^2 f(x_k) - B_k\|^3}{\sigma_{k}^3} $ since a bound on $\frac{\beta_k^3}{\sigma_{k}^3}$ has been derived in \eqref{crucialhess}. Using now \eqref{condTkapprox} and the definition of  \eqref{xikdef}
	\begin{align*}
		4 \frac{\|\nabla_x^2 f(x_k) - B_k\|^3}{\sigma_{k}^3} &\leq \frac{4 \kappa_D}{\sigma_{k}^3} \xi_k \leq  \frac{4 \kappa_D}{\sigma_{k}^3} \sum_{j=k-2m+1}^{k-1} \frac{\sigma_{j+1} - \sigma_{j}}{\sigma_{j}} \\
		&\leq \frac{4 \kappa_D}{\sigma_{k}^2 \sigma_{0} } \sum_{j=k-2m+1}^{k-1} \frac{\sigma_{j+1} - \sigma_{j}}{\sigma_{j+1}} \leq  \frac{2^{2m+1} \kappa_D }{\sigma_{0}^3 } \sum_{j=k-2m+1}^{k-1} \frac{\sigma_{j+1} - \sigma_{j}}{\sigma_{j+1}}.
	\end{align*}
	Where we used the fact that $\sigma_{j} \geq \frac{\sigma_{0}}{2^{-2m+1}}$ for $j \geq -2m+1$ and that $\sigma_{k} \geq \sigma_{0}$ for $k\geq0$. Including the last inequality in \eqref{firststep} and using \eqref{crucialhess}, we obtain \eqref{xikbound} for $p=2$.
	Remark that \eqref{xikbound} still holds for $p \geq 3$ since \eqref{crucialhess} applies and $\chi_k = \beta_k$.

\end{proof}

\subsection{Proof of Lemma~\ref{sumsjbound}}\label{fourthproof}

\begin{proof}
	From \eqref{skboundneg} and \eqref{skboundizycase}, we derive that,
	\begin{align*}
		\|s_k\|^{p+1} &\leq \|s_k\|^{p+1} \indica{\|s_k\| \leq 2 \eta} + \|s_k\|^{p+1} \indica{\|s_k\| \leq 2 (\frac{(p+1)! \|g_k\|}{\sigma_{k}})^\sfrac{1}{p} } + \|s_k\|^{p+1} \indica{\|s_k\| \leq 2 (\frac{(p+1)! \chi_k}{2 \sigma_{k}} )^\sfrac{1}{p-1} } \\
		&\leq (1+(2\eta)^{p+1}) \frac{\sigma_{j+1}-\sigma_{j}}{\sigma_{j+1}}  +  2^{p+1} \left(\frac{(p+1)! \|g_k\|}{\sigma_{k}}\right)^\sfrac{p+1}{p} + 2^{p+1} \left(\frac{(p+1)! \chi_k}{2 \sigma_{k}}\right)^\sfrac{p+1}{p-1}.  \\
	\end{align*}
	Now from \cite[Lemma~11]{GraJerToin24}, the sum $\sum_{j=0}^{k} (1+(2\eta)^{p+1}) \frac{\sigma_{j+1}-\sigma_{j}}{\sigma_{j+1}}  +  2^{p+1} \left(\frac{(p+1)! \|g_k\|}{\sigma_{k}}\right)^\sfrac{p+1}{p} $ has already been handled and so we know that there exists $\kappa_o$ and $\kappa_i$ such that 
	\begin{equation}\label{sumaltreated}
		\sum_{j=0}^{k}   (1+(2\eta)^{p+1}) \frac{\sigma_{j+1}-\sigma_{j}}{\sigma_{j+1}}  +  2^{p+1} \left(\frac{(p+1)! \|g_k\|}{\sigma_{k}}\right)^\sfrac{p+1}{p} \leq \kappa_o + \kappa_i \log(\sigma_{k+1}),
	\end{equation}
	with the expression of $\kappa_o$ and $\kappa_i$ can be found in \cite[Lemma~11]{GraJerToin24} with the appropriately chosen $\kappa_D$ and $m$. We now turn to $\bigsum_{j=0}^{k} \frac{\powppminsone{\chi_j}}{\powppminsone{\sigma_{j}}}$. Using \eqref{xikzero} if $j=0$ and \eqref{xikbound} otherwise, we derive that
	\begin{align}
		\sum_{i=0}^{k} \frac{\powppminsone{\chi_i}}{\powppminsone{\sigma_{i}}} &\leq 
		4 \left( \frac{\max[0,-\lambda_{\min}(\nabla_x^2 f(x_0))]^\sfrac{p+1}{p-1}}{\sigma_{0}^\sfrac{p+1}{p-1}} + \frac{m^3 \kappa_A^3 \mathds{1}_{p=2}}{\sigma_{0}^3} \right) \nonumber
		\\
		&+ \sum_{i=1}^{k} \frac{\kappa_D 2^{2m+1} \indica{p=2}}{\sigma_{0}^3} \sum_{j=i-2m+1}^{i-1} \frac{\sigma_{j+1}-\sigma_{j}}{\sigma_{j+1}} 
		\nonumber  \\ 
		&+\sum_{i=1}^{k}  	4 \kappa_3 \frac{\sigma_{i}-\sigma_{i-1}}{\sigma_{i}} + 4 \kappa_4 \sum_{j=i-2m}^{i-2} \frac{\sigma_{j+1}-\sigma_{j}}{\sigma_{j+1}}. \label{sumtobound}
	\end{align}
	
	We now provide a bound on the two sums involving
	${\frac{\sigma_{j+1} -\sigma_{j}}{\sigma_{j}}}$.
	By inverting the two sums,
	Lemma~\ref{ajsum} and the fact that $\sigma_{k}$ is non-decreasing,
	we derive after some simplification,  that 
	\begin{align}\label{firstboundellj}
		\sum_{i= 1}^k \sum_{j = i-2m+1}^{i-1} {\frac{\sigma_{j+1} - \sigma_{j}}{\sigma_{j+1}}} &= \sum_{j = 1}^{2m-1} \sum_{i= 1}^k  {\frac{\sigma_{i-j+1} - \sigma_{i-j}}{\sigma_{i-j+1}}} \nonumber \\
		&\leq  \sum_{j = 1}^{2m+1} {\log(\sigma_{k-j+1}) -
			\log(\sigma_{1-j})} \leq (2m-1) {\log({\sigma_k})} - \sum_{j=1}^{2m-1}
		\log\left(\frac{\sigma_{0}}{2^{1-j}}\right) \nonumber \\ 
		&\leq 2m ({\log(\sigma_{k})} - \log(\sigma_{0})) +m (2m+1) \log(2) .
	\end{align}
	Similarly for $\sum_{i=1}^{k} \sum_{j=i-2m}^{j-2}
	{\frac{\sigma_{j+1} - \sigma_{j}}{\sigma_{j+1}}}$ and using
	the same arguments as  above yields that
	\begin{equation}\label{secboundellj}
		\sum_{i= 1}^k \sum_{j = i-2m}^{i-2} {\frac{\sigma_{j+1} - \sigma_{j}}{\sigma_{j+1}}} = \sum_{j = 1}^{2m-1} \sum_{i= 1}^k  {\frac{\sigma_{i-j} - \sigma_{i-j-1}}{\sigma_{i-j}}} \leq 2m({\log(\sigma_{k-1})} - \log(\sigma_{0})) +m(2m+1) \log(2).
	\end{equation}
	
	Now injecting \eqref{firstboundellj}, \eqref{secboundellj} in \eqref{sumtobound} and using once again Lemma~\ref{ajsum}, we derive that
	
	\begin{align*}
		\sum_{i=0}^{k} \frac{\powppminsone{\xi_i}}{\powppminsone{\sigma_{i}}} &\leq 
		4 \left( \frac{\max[0,-\lambda_{\min}(\nabla_x^2 f(x_0))]^\sfrac{p+1}{p-1}}{\sigma_{0}^\sfrac{p+1}{p-1}} + \frac{m^3 \kappa_A^3 \mathds{1}_{p=2}}{\sigma_{0}^3} \right) \\	
		&+ \frac{\kappa_D 2^{2m+1}\indica{p=2}}{\sigma_{0}^3} \left( 2m \log(\sigma_{k})- 2m \log(\sigma_{0}) +m(2m+1) \log(2)\right) \\
		&+ 4 \kappa_3 (\log(\sigma_{k}) - \log(\sigma_{0} )) + 4 \kappa_4 \left( 2m \log(\sigma_{k-1})-2m\log(\sigma_{0}) +m(2m+1) \log(2) \right)	
	\end{align*}
	Now combining the last inequality with the result stated in \eqref{sumaltreated} and using that $\sigma_k$ is a non-decreasing sequence, we derive the result stated in \eqref{exprsumsjbound} with $\kappa_{const}$ and $\kappa_{log}$ being constants depending on the problem. 
\end{proof}

\subsection{Proof of Lemma~\ref{sigmakbound}}\label{secondappendix}
\begin{proof}
	From \cite[{Lemma~6}]{GraJerToin24} with  $\kappa_D$ defined as \eqref{kappaDdef},  we know that, for $j\geq0$
	\begin{equation}\label{locdecr}
		\frac{\sigma_{j} \|s_j\|^{p+1}}{(p+1)!} \leq f(x_j) - f(x_{j+1}) + \kappa_a \|s_j\|^{p+1} + \kappa_D \alpha_p \xi_j
	\end{equation}
	with 
	\begin{equation*}
		\kappa_a \eqdef \frac{L_p}{(p+1)!} +  \sum_{i=1}^p \frac{i}{i!(p+1)} \quad \alpha_p \eqdef \sum_{i=1}^p \frac{p+i-1}{i!(p+1)}
	\end{equation*}

	Let $j \in \iibe{0}{k}$.	First see that from \eqref{sigmajnegdef} and the definition of $\xi_j$, we have that
	\begin{align*}
		\sum_{j=0}^{k} {\xi_j} = \sum_{j=0}^{k} \sum_{i=1}^{2m-1} {\|s_{j-i}\|^{p+1}} &= \sum_{i=1}^{2m-1} \sum_{j=0}^{i-1} {\|s_{j-i}\|^{p+1}} + \sum_{i=1}^{2m-1} \sum_{j=i}^{k} {\|s_{j-i}\|^{p+1}} \\
		&\leq \sum_{i=1}^{2m-1} i + \sum_{i=1}^{2m-1} \sum_{j=0}^{k-1} {\|s_{j}\|^{p+1}} \\
		&\leq m(2m-1) + (2m-1) \sum_{j=0}^{k-1} {\|s_{j}\|^{p+1}}
	\end{align*}
	
	Summing \eqref{locdecr} for all $j$ and using the previous inequality to bound $\sum_{j=0}^{k} {\xi_j}$, we derive that 
	\begin{align*}
		{\sum_{j= 0}^k \frac{\sigma_{j}}{(p+1)!} \|s_j\|^{p+1}} 
		\leq & \sum_{j=0}^{k} f(x_{j}) - f(x_{j+1}) + \sum_{j= 0}^{k} \kappa_a {\|s_j\|^{p+1}} \\ 
		&+ \kappa_D  \alpha_p {\left(m(2m-1) +(2m-1) \sum_{j=0}^{k-1} {\|s_j\|^{p+1}}\right)}.
	\end{align*}
	
	Using \eqref{sigmakupdate} to simplify the left-hand side, Assumption~\ref{assum2}, and using \eqref{exprsumsjbound} to bound both
	$\sum_{j=0}^k {\|s_j\|^{p+1}}$ and $\sum_{j=0}^{k-1}
	{\|s_j\|^{p+1}}$, we obtain that 
	\begin{align*}
		\frac{\sigma_{k+1} - \sigma_{0} }{(p+1)!} 
		\leq & f(x_0) - f_{\rm low} +  \kappa_{a} (\kappa_{const} + \kappa_{log} \log({\sigma_{k+1}}) \\ 
		&+ (2m-1)\kappa_D \alpha_p\left(m+\kappa_{const} +\kappa_{log} \log({\sigma_{k}})\right).
	\end{align*}
	The fact that the $\sigma_j$ sequence is increasing, the last
	inequality gives that
	\begin{align}\label{Sigmakineq}
		{\frac{\sigma_{k+1}}{(p+1)!}} &\leq  f(x_0) - f_{\rm low} + \frac{\sigma_{0}}{(p+1)!} + \kappa_a (\kappa_{const} + \kappa_{log} \log({\sigma_{k+1}})) \nonumber \\ 
		&+ \kappa_D \alpha_p (2m-1) ( m +\kappa_{const} +\kappa_{log} \log(\sigma_{k+1})).
	\end{align}  
	Now define
	\begin{align}\label{ggu}
		\gamma_1 &\eqdef \kappa_{a} \kappa_{log} + (2m-1) \kappa_D \alpha_p \kappa_{log},
		\, \, \, \gamma_2 \eqdef -\frac{1}{(p+1)!},
		\, \, \, u \eqdef {\sigma_{k+1}}, \nonumber \\
		\gamma_3 &\eqdef f(x_0) - f_{\rm low} + \frac{\sigma_{0}}{(p+1)!} + \kappa_{a}\kappa_{const} +(2m-1) \kappa_D\alpha_p(m + \kappa_{const}) 
	\end{align}
	and observe that that $-3 \gamma_2 <  \gamma_1$ since
	$(p+1)! \kappa_{a} \geq L_p \geq 3$ and $\kappa_{log} \geq 1$.
	Define the function \\
	$\psi(t) \eqdef \gamma_1 \log(t) + \gamma_2 t + \gamma_3$.
	The inequality \eqref{Sigmakineq} can then be rewritten as
	\begin{equation}\label{psirewrite}
		0 \leq \psi(u).
	\end{equation}
	The constants $\gamma_1$, $\gamma_2$ and $\gamma_3$ satisfy the
	requirements of \cite[Lemma~B.1]{GraJerToin24} and $\psi$
	therefore admits two roots $\{u_1,u_2\}$ whose expressions are given
	in \cite[equation (B.2)]{GraJerToin24}. Moreover, \eqref{psirewrite} is valid only for
	$u \in [u_1, \, u_2]$. Therefore, we obtain from  \cite[Lemma~B.1]{GraJerToin24} that
	\begin{equation}\label{sigmamaxrxpr}
		{\sigma_{k+1}} \leq \sigma_{\max} \eqdef -(p+1)!\gamma_1 W_{-1}\left( \frac{-1}{(p+1)!\gamma_1} e^{-\frac{\gamma_3}{\gamma_1}}\right),
	\end{equation}
	{where $W_{-1}$ is the second branch of the Lambert function}
	We then derive the desired result because the last
	inequality holds for all $k \geq 0$ and $\sigma_{k}$ is increasing.
\end{proof}

\section{Proof of Lemma~\ref{constructWk}}\label{Proofconstructwk}
Since we just need to construct $W$, we will focus on the case where $W$ is symmetric positive definite. Denote by $A = W^{-2}$. Our subsequent analysis will be divided in two cases. We will suppose in the following that $s^\intercal y > 0$ as we can consider $-s$ instead of $s$ in \eqref{HOSU} to enforce  $s^\intercal y > 0$ and the obtained $W$ is going to yield the same update rule as DFP.

\paragraph{(s,y) co-linear}
Suppose that $(s,y)$ are co-linear. Define the orthogonal basis $V$ such that
\begin{equation}\label{Vfirstcase}
	v_1 = \frac{s}{\|s\|}, \quad (v_2, \cdots , v_n).
\end{equation}
And define $W$ the as the matrix
\begin{equation}\label{Wfirstcase}
	W = V 
	\begin{pmatrix}
		&\frac{\|s\|}{\sqrt{s^\intercal y}}& &  & &\\
		&  & 1 &        &  &\\
		&  &   & \ddots &  & \\
		&  &   &        & 1 &
	\end{pmatrix}
	V^T
\end{equation}
Note that $W$ is positive definite since $	\frac{\|s\|}{\sqrt{s^\intercal y}} > 0$, and see that
\[
W^{-2} s = \|s\| W^{-2} v_1 = \|s\| \frac{s^\intercal y}{\|s\|^2} v_1 =  \frac{s^\intercal y}{\|s\|^2} s = y
\]
where the last inequality is implied by the fact that $(s,y)$ are co-linear. Using that $\mu \|s\|^2 \leq s^\intercal y \leq L \|s\|^2$, we derive
\begin{equation}\label{borneWfirstcase}
	\kappa(W) \leq \max[\sqrt{\frac{1}{L}} , \sqrt{\frac{1}{\mu}}] \leq \sqrt{\frac{1}{\mu}}.
\end{equation}
We now turn to the case where $(s,y)$ are not colinear.

\paragraph{(s,y) linearally independent}
We now construct the orthonormal basis $V$ such that 
\begin{equation}\label{Vseccase}
	v_1 = \frac{s}{\|s\|}, \quad v_2 = \frac{ y - s^\intercal y \frac{s}{\|s\|^2}}{\lvert\lvert y - s^\intercal y \frac{s}{\|s\|^2} \rvert\rvert}, \quad (v_3, \cdots , v_n).
\end{equation}
Now, if we choose a matrix $A = W^{-2}$ such that
\begin{equation}\label{Aseccase}
	A =  V 
	\begin{pmatrix}
		& a & c &   &        & &    \\
		& c & b &   &        & &  \\ 
		&   &   & 1 &        & &  \\
		&   &   &   & \ddots & & \\
		&   &   &   &  &       1& \\
	\end{pmatrix}
	V^T
\end{equation}
With the conditions that $a > 0$ and $ab-c^2 \geq \varsigma$ so that  $W = A^{-1/2}$ can be reconstructed and  have a bounded condition number.
Since from \eqref{HOSU}, we have that $A s = y$ and so from both \eqref{Aseccase} and \eqref{Vseccase}, we obtain that 
\[
y = a \|s\| v_1 + c \|s\| v_2 \text{ with } y = \lvert\lvert y - s^\intercal y \frac{s}{\|s\|^2} \rvert\rvert v_2 + \frac{s^\intercal y}{\|s\|} v_1. 
\]
By identification we have that, 
\begin{equation}\label{aandcchoice}
	a = \frac{s^\intercal y}{\|s\|^2}, \quad c = \frac{\lvert\lvert y - s^\intercal y \frac{s}{\|s\|^2} \rvert\rvert}{\|s\|} 
\end{equation}
Now using the bound $ab-c^2 \geq \varsigma$, we can choose $b$ such that
\begin{equation}\label{bchoice}
	b = \frac{1}{a}(\varsigma + c^2) = \frac{1}{a} \left(\varsigma +  \frac{\|y\|^2 - \frac{(s^\intercal y)^2}{\|s\|^2}}{\|s\|^2}\right) 
\end{equation}
Note that since $\mu \|s\|^2 \leq s^\intercal y \leq L \|s\|^2 $ and $\|y\|^2 \leq L^2 \|s\|^2$ hold and so we bound  \eqref{aandcchoice} and \eqref{bchoice} as,
\begin{equation}\label{boundsonb}
	\mu \leq a  \leq L, \quad b \leq \frac{\varsigma + L^2}{\mu}. 
\end{equation}
Now denote by $B = \begin{pmatrix}
	a  & c \\
	c  & b
\end{pmatrix}$ and  by $\lambda_1$ and $\lambda_2$ its eigenvalue with $\lambda_1 \leq \lambda_2$. Since $ab-c^2 = \varsigma$ and from \eqref{boundsonb}, we know that 
\begin{equation}\label{lambdaboundss}
	\mu \leq a \leq \lambda_2 \leq (a+b) \leq L + \frac{\varsigma + L^2}{\mu} \quad \text{ and } \frac{1}{\lambda_1} = \frac{\lambda_2}{ab-c^2} \leq \frac{\lambda_2}{\varsigma}
\end{equation}

Now from \eqref{Aseccase}, \eqref{lambdaboundss}, that $B = \begin{pmatrix}
	a  & c \\
	c  & b
\end{pmatrix}$ and  $A = W^{-2}$, we have that

\begin{align}\label{borneWseccase}
	\kappa(W) = \sqrt{\kappa(A)} &\leq \sqrt{\max \left[\frac{1}{\lambda_1} , \lambda_2, \frac{\lambda_2}{\lambda_1}\right]} \nonumber \\
	&\leq  \sqrt{\max \left[\frac{\lambda_2}{\varsigma}, \lambda_2, \frac{\lambda_2^2}{\varsigma} \right] } \nonumber \\
	&\leq  \sqrt{\lambda_2}  \sqrt{\max \left[\frac{1}{\varsigma}, 1, \frac{\lambda_2}{\varsigma} \right] } \nonumber \\
	&\leq \sqrt{ L + \frac{\varsigma + L^2}{\mu}} \sqrt{\max\left[\frac{1}{\varsigma}, \, 1, \, \frac{\varsigma + L^2}{\mu\varsigma}\right]}
\end{align}

Now combining both \eqref{borneWfirstcase} and \eqref{borneWseccase} yields the desired result.

%USE THE BELOW OPTIONS IN CASE YOU NEED AUTHOR YEAR FORMAT.
%\bibliographystyle{abbrvnat}
%\bibliography{reference}

\end{document}